\documentclass[11pt,a4paper,reqno]{amsart}
\usepackage{amsfonts,amssymb,amsmath,amsthm,graphicx,color,amscd,xspace,verbatim}
\usepackage{scalerel}
\usepackage{constants}
\usepackage{url}
\makeatletter

\usepackage{hyperref}
\makeatletter
\def\@tocline#1#2#3#4#5#6#7{\relax
  \ifnum #1>\c@tocdepth % then omit
  \else
    \par \addpenalty\@secpenalty\addvspace{#2}%
    \begingroup \hyphenpenalty\@M
    \@ifempty{#4}{%
      \@tempdima\csname r@tocindent\number#1\endcsname\relax
    }{%
      \@tempdima#4\relax
    }%
    \parindent\z@ \leftskip#3\relax \advance\leftskip\@tempdima\relax
    \rightskip\@pnumwidth plus4em \parfillskip-\@pnumwidth
    #5\leavevmode\hskip-\@tempdima
      \ifcase #1
       \or\or \hskip 1em \or \hskip 2em \else \hskip 3em \fi%
      #6\nobreak\relax
      \dotfill
      \hbox to\@pnumwidth{\@tocpagenum{#7}}
    \par
    \nobreak
    \endgroup
  \fi}
\makeatother

\newconstantfamily{c}{symbol=c}

\newtheorem{theorem}{Theorem}[section]
\newtheorem{lemma}[theorem]{Lemma}
\newtheorem{proposition}[theorem]{Proposition}

\theoremstyle{definition}
\newtheorem{definition}[theorem]{Definition}

\theoremstyle{remark}

\newcommand{\R}{{\mathbb R}}

\newcommand{\beqn}{\begin{eqnarray}}
\newcommand{\eeqn}{\end{eqnarray}}   % equation with number
\newcommand{\beq}{\begin{eqnarray*}}
\newcommand{\eeq}{\end{eqnarray*}}

\newcommand{\be}{\small\begin{equation}}
\newcommand{\bel}[1]{\begin{equation}\label{#1}}
\newcommand{\ee}{\end{equation}\normalsize}%% This macro does not work with amstex.
\newcommand{\BA}{\begin{array}}
\newcommand{\EA}{\end{array}}
\newcommand{\BAN}{\renewcommand{\arraystretch}{1.2}
\setlength{\arraycolsep}{2pt}\begin{array}}
\newcommand{\BAV}[2]{\renewcommand{\arraystretch}{#1}
\setlength{\arraycolsep}{#2}\begin{array}}
\newcommand{\BSA}{\begin{subarray}}
\newcommand{\ESA}{\end{subarray}}
\newcommand{\BAL}{\begin{aligned}}
\newcommand{\EAL}{\end{aligned}}

\newcommand{\forevery}{\quad \forall}

\newcommand{\norm}[1]{\left \|#1\right \|}%% adjustable norm
\newcommand{\supp}{\mathrm{supp}\,}
\newcommand{\dist}{\mathrm{dist}\,}
\newcommand{\sign}{\mathrm{sign}}
\newcommand{\diam}{\mathrm{diam}\,}
\newcommand{\sbs}{\subset}

\def\ga{\alpha}

            \def\gl{\lambda}
\def\gm{\mu}                 
            
       \def\gt{\tau}
      \def\gw{\omega}

\def\Gl{\Lambda}          
\def\Gw{\Omega}              

   \def\CO{{\mathcal O}}   
      
      \def\CF{{\mathcal F}}

\def\BBG {\mathbb G}       
       
   \def\BBN {\mathbb N}    
   \def\BBR {\mathbb R}    \def\BBS {\mathbb S}
\def\BBT {\mathbb T}

\newcommand{\xa}{\alpha}
\newcommand{\xb}{\beta}
\newcommand{\xg}{\gamma}

\newcommand{\xd}{\delta}
\newcommand{\xD}{\Delta}
\newcommand{\xe}{\varepsilon}
\newcommand{\xz}{\zeta}

\newcommand{\xk}{\kappa}
\newcommand{\xl}{\lambda}
\newcommand{\xL}{\Lambda}
\newcommand{\xm}{\mu}
\newcommand{\xn}{\nu}

\newcommand{\xr}{\rho}
\newcommand{\xs}{\sigma}

\newcommand{\xf}{\phi}
\newcommand{\xF}{\Phi}

\newcommand{\xO}{\Omega}

\newcommand{\myfrac}[2]{{\displaystyle \frac{#1}{#2} }}

\def\bal#1\eal{\small\begin{align*}#1\end{align*}\normalsize}
\def\ba#1\ea{\small\begin{align}#1\end{align}\normalsize}
\numberwithin{equation}{section}

\def \dd {\mathrm{d}}

\def\Xint#1{\mathchoice
{\XXint\displaystyle\textstyle{#1}}%
{\XXint\textstyle\scriptstyle{#1}}%
{\XXint\scriptstyle\scriptscriptstyle{#1}}%
{\XXint\scriptscriptstyle\scriptscriptstyle{#1}}%
\!\int}
\def\XXint#1#2#3{{\setbox0=\hbox{$#1{#2#3}{\int}$}
\vcenter{\hbox{$#2#3$}}\kern-.5\wd0}}

\def\dashint{\Xint-}
\newcommand{\1}{\textbf{1}}

\DeclareMathOperator*{\essliminf}{ess\,lim\,inf}

\DeclareMathOperator*{\essinf}{ess\,inf}

\begin{document}
\title[Nonlinear nonlocal equations]{Nonlinear nonlocal equations involving subcritical or power nonlinearities and measure data }%%
%%}%% Title to be inserted !!!!!!!!!!
\author{Konstantinos T. Gkikas}
\address{Konstantinos T. Gkikas, Department of Mathematics, University of the Aegean, 832 00 Karlovassi, Samos,
Greece\newline
Department of Mathematics, National and Kapodistrian University of Athens, 15784 Athens, Greece.}
\email{kgkikas@aegean.gr}

\date{\today}

\begin{abstract}
Let $s\in(0,1),$ $1<p<\frac{N}{s}$ and $\xO\subset\BBR^N$ be an open bounded set. In this work we study the existence of solutions to problems ($E_\pm$) $Lu\pm g(u)=\xm$ and $u=0$ a.e. in $\BBR^N\setminus\xO,$ where $g\in C(\BBR)$ is a nondecreasing function, $\xm$ is a bounded Radon measure on $\xO$ and $L$ is an integro-differential operator with order of differentiability $s\in(0,1)$ and summability $p\in(1,\frac{N}{s}).$ More precisely, $L$ is a fractional $p-$Laplace type operator. We establish sufficient conditions for the solvability of problems ($E_\pm$). In the particular case $g(t)=|t|^{\xk-1}t;$ $\xk>p-1,$ these conditions are expressed in terms of Bessel capacities.

	\medskip
	
	
	\medskip
	
	
\end{abstract}

\maketitle

\date{\today}

\tableofcontents
\section{Introduction}
Let $\xO\subset\BBR^N$ be an open bounded domain, $s\in(0,1)$ and $1<p<\frac{N}{s}.$ In this article we are concerned with the existence of very weak solutions to the quasilinear nonlocal problems

\ba\tag{$P_\pm$}\label{main}
\left\{
\BAL
L u \pm g(u)&=\xm\quad&&\text{in}\;\; \xO\\
u&=0\quad &&\text{in}\;\; \BBR^N\setminus\xO,
\EAL\right.
\ea
where $\xm$ is a bounded Radon measure on $\xO$ and $g\in C(\BBR)$ is a nondecreasing function such that $g(0)=0.$ Here,  the nonlocal operator $L$ is defined by
\bal
Lu(x):=P.V.\int_{\BBR^N}|u(x)-u(y)|^{p-2}(u(x)-u(y))K(x,y){\dd} y\quad\forall x\in\xO,
\eal
where the symbol P.V. stands for the principle value integral and $K:\BBR^N\times\BBR^N\to \BBR$  is a measurable and symmetric (i.e. $K(x,y)=K(y,x)$) function. Note that if $K(x,y)\equiv|x-y|^{-N-sp}$ then $L$ coincides with the standard fractional $p$-Laplace operator $(-\xD)^s_p.$ 

Throughout this work, we assume that there exists a positive constant $\xL_K\geq1$ such that the following ellipticity condition holds 
\bal
\xL_K^{-1}|x-y|^{-N-sp}\leq K(x,y)\leq \xL_K|x-y|^{-N-sp}\qquad \forall (x,y)\in \BBR^{N}\times\BBR^N\;\;\text{and} \;\;x\neq y.
\eal
In addition, we denote by $\mathfrak{M}_b(\xO)$ the space of Radon measures on $\BBR^N$ such that $\xm(\BBR^N\setminus\xO)=0,$ as well as by $\mathfrak{M}^+_b(\xO)$ its positive cone.

Let 

\bal C_{N,s}:=2^{2s}\pi^{-\frac N2}s\frac{\Gamma(\frac{N+2s}2)}{\Gamma(1-s)}>0.
\eal
For $p=2$ and $K(x,y)=C_{N,s}|x-y|^{-N-2s},$  operator $L$ reduces to the well-known fractional Laplace operator $(-\xD)^s$ and the problem $P_+$ becomes  

\ba\label{abslinear}
\left\{
\BAL
(-\xD)^su + g(u)&=\xm\quad&&\text{in}\;\; \xO\\
u&=0\quad &&\text{in}\;\; \BBR^N\setminus\xO.
\EAL\right.
\ea
When $g$ satisfies the subcritical integral condition

\ba\label{condp=2}
\int_1^\infty (g(s)-g(-s))s^{-\frac{N}{N-2s}-1}{\dd} s<\infty,
\ea
Chen and V\'{e}ron \cite{CV1} showed that problem \eqref{abslinear} admits a unique very weak solution for any $\xm\in \mathfrak{M}_b(\xO).$  In addition they showed that problem \eqref{abslinear} with $g(u)=|u|^{\xk-1}u\; (\xk>1)$ possesses a very weak solution if and only if $\xm$ is absolutely continuous with respect  to Bessel capacity $C_{L_{2s,\xk'}},$ i.e. $\xm$ vanishes on compact set $E$ of $\xO$ satisfying $\mathrm{Cap}_{{2s,\xk'}}(E) = 0$
(see \eqref{besselcapacity} for the definition of the Bessel capacities). Their approach is based on the properties of the Green Kernel associated with fractional Laplace operator $(-\xD)^s$ in $\xO.$

In the local theory and more precisely when $Lu=-\xD_p u=-\text{div}(|\nabla u|^{p-2}\nabla u),$ related problems have been studied in \cite{B,BNV,BG,egw,PVsource,PV,Vbook}. In particular, in the power case, i.e.

\ba\label{powerabsorptionlocal}
\left\{\BAL
-\xD_p u+|u|^{\xk-1}u&=\xm\quad&&\text{in}\;\; \xO\\
u&=0\quad &&\text{on}\;\; \partial\xO,
\EAL\right.
\ea
Bidaut-V\'eron, Nguyen and V\'eron \cite{BNV} established that if $\xm\in \mathfrak{M}_b(\xO)$ is absolutely continuous with respect to the Bessel capacity $\mathrm{Cap}_{{p,\frac{\xk}{\xk-p+1}}}$, then there exists a renormalized solution to problem \eqref{powerabsorptionlocal} with $\xk>p-1.$ A main ingredient in the proof of this result is the pointwise estimates for  $p$-superharmonic functions in $\xO.$ These pointwise estimates are expressed in terms of the truncated Wolff potentials $W_{1,p}^{R}[\xm]$ (see, e.g., \cite{HKM,KM1,KM2,PV}). We recall here that the truncated Wolff potential is given by 

\ba\label{wolfftr}
W_{\xa,p}^R[\xm](x):=\int_0^{R}\left(\frac{|\xm|(B_r(x))}{r^{N-\xa p}}\right)^{\frac{1}{p-1}}\frac{{\dd} r}{r}, 
\ea
for any $R>0$ and $\xa\in (0,N)$ such that $p\in(1,\frac{N}{\xa}).$  Conversely, Bidaut-V\'eron \cite{B} showed that if problem \eqref{powerabsorptionlocal} with $\xk>p-1$ admits a renormalized solution, then $\xm$ is absolutely continuous with respect to the Bessel capacity $\mathrm{Cap}_{{p,\frac{\xk}{\xk-p+1}+\xe}}$, for any $\xe>0.$

Phuc and Verbitsky \cite{PV} showed that if $\tau\in\mathfrak{M}_b^+(\xO)$ has compact support in $\xO,$ then the problem
\ba\label{sourcpowerintro}
\left\{\BAL
-\xD_p u -|u|^\xk&=\xr\tau\quad&&\text{in}\;\; \xO\\
u&=0\quad&&\text{on}\;\; \partial\xO,
\EAL\right.
\ea
admits a nonnegative renormalized solution for some $\xr>0,$ if and only if, there exists a positive constant $C$ such that

\ba\label{sourcecond}
\tau(K)\leq C \mathrm{Cap}_{{p,\frac{\xk}{\xk-p+1}}}(K),
\ea
for any compact $K\subset\xO.$ Moreover, they showed that \eqref{sourcecond} is equivalent to

\bal
W_{1,p}^{2\diam(\xO)}[(W_{1,p}^{2\diam(\xO)}[\tau])^\xk]\leq CW_{1,p}^{2\diam(\xO)}[\tau] \quad \text{a.e. in}\;\;\xO,
\eal
for some positive constant $C>0.$

Recently, a great attention has been drawn to the study of  the fractional $p-$Laplacian or more general nonlocal operators (see for example \cite{abde,diCKP_local,diCKP_local2,IMS,KLL,KKL,KKP,KKP2,KMS2,KMS,LL,LL2,P} ). More precisely, Kuusi, Mingione and Sire \cite{KMS} dealt with the problem

\ba\label{linearweight}
\left\{
\BAL
L_\xF u&=\xm\quad&&\text{in}\;\; \xO\\
u&=g\quad &&\text{in}\;\; \BBR^N\setminus\xO,
\EAL\right.
\ea
where $g\in W^{s,p}(\BBR^N),$ $L_\xF$ is a nonlocal operator defined by

\bal
<L_\xF u,\xz>:=\int_{\BBR^N}\int_{\BBR^N}\xF(u(x)-u(y))(\xz(x)-\xz(y))K(x,y){\dd} y{\dd} x\quad\forall \xz\in C_0^\infty(\xO).
\eal
Here  $\xF:\BBR\to\BBR$ is a continuous function such that $\xF(0)=0$ and

\bal
\xL_\xF^{-1}|t|^{p}\leq \xF(t)t\leq \xL_\xF|t|^{p}.
\eal
When $2-\frac{s}{N}<p,$ they show the existence of a very weak solution to \eqref{linearweight}, which they called SOLA (Solutions obtained as limits of approximations). They also showed local pointwise estimates for SOLA to \eqref{linearweight} in terms of the truncated Wolff Potential $W_{s,p}^{R}[\xm].$ In the particular case $\xF(t)=|t|^{p-2}t$ and $g=0,$ the existence of very weak solutions was established in \cite{abde} for any $1<p<\frac{N}{s}.$

The objective of this work is to determine the subcritical integral conditions on $g,$ which ensure the existence of very weak solutions to problems \eqref{main}. In addition, in the power case, i.e. $g(u)=|u|^{\xk-1}u;$ $\xk>p-1,$ we aim to find sufficient conditions, expressed in terms of Bessel capacities like above, for the solvability of \eqref{main}.

Let us mention here that our work is inspired by the article \cite{BNV} for problem ($P_+$) and by the articles \cite{PVsource,PV} for problem ($P_-$) with $g(u)=|u|^{\xk-1}u;\;\xk>p-1$. However, due to the presence of the nonlocal operator, new essential difficulties arise which complicates drastically the study of  problems \eqref{main}.

In order to state our main results, we need to introduce the notion of the very weak solutions.

\begin{definition}\label{definsolu}
Let $s\in(0,1),$ $1<p<\frac{N}{s},$ $\tilde g\in C(\BBR),$ $\xO\subset\BBR^N$ be an open bounded domain and $\xm\in \mathfrak{M}(\xO).$  We will say that $u:\BBR^N\to\BBR$ is a very weak solution to the problem

\ba\label{absorption}
\left\{
\BAL
L u +\tilde g(u)&=\xm\quad&&\text{in}\;\; \xO\\
u&=0\quad &&\text{in}\;\; \BBR^N\setminus\xO,
\EAL\right.
\ea
if  $\tilde g(u)\in L^1_{loc}(\xO)$ and if the following conditions are valid:

(i) $u=0$ a.e. in $\BBR^N\setminus\xO$ and $u\in W^{h,q}(\BBR^N)$ for any $0<h<s$ and for any $0<q<\frac{N(p-1)}{N-s}.$

(ii) $T_k(u):=\max(-k,\min(u,k))\in W_0^{s,p}(\xO)$ for any $k>0.$ 

(iii) 

\bal
\int_{\BBR^N}\int_{\BBR^N}|u(x)-u(y)|^{p-2}(u(x)-u(y))(\xf(x)-\xf(y))K(x,y){\dd} x{\dd} y&+\int_\xO\tilde g(u)\xf {\dd} x=\int_\xO\xf {\dd}\xm
\eal
for any $\xf\in C_0^\infty(\xO).$
\end{definition}
\noindent We note here that if $2-\frac{s}{N}<p<\frac{N}{s},$ then the very weak solution $u$ belongs to the fractional Sobolev space $W^{h,q}(\BBR^N)$ for any $q \in (1,\frac{N(p-1)}{N-s}).$ If $p\leq 2-\frac{s}{N},$ the space $W^{h,q}(\BBR^N)$ in the above definition is no longer a fractional Sobolev space, however it is defined in the same way (see \eqref{sobolev}).   
 
In Section 2, we discuss the existence and main properties of the very weak solutions of problem \eqref{absorption} with $\tilde g\equiv0.$ Particularly, in the spirit of \cite{KMS}, we show the existence of a SOLA $u$ satisfying statements (i)-(iii) of the above definition (see Proposition \ref{sola}). The approximation sequence consists of solutions of \eqref{absorption} with $\tilde g\equiv0$ and smooth data. In addition, we prove that these solutions satisfy a priori estimates \eqref{est7} and \eqref{est1}. As a result, we establish that the very weak solution satisfies \eqref{est1} and
\ba\label{weakest}
\norm{|u|^{p-1}}^{*}_{L^{\frac{N}{N-sp}}_w(\BBR^N)}\leq C(N,p,s,\xL_K) \int_\xO|\xm|{\dd} x,
\ea
where $\norm{\cdot}^{*}_{L^{\frac{N}{N-sp}}_w(\BBR^N)}$ has been defined in \eqref{semi} and is related to the Marcinkiewicz spaces. Finally, when $\xm\in \mathfrak{M}_b^+(\xO),$ we construct this solution (see Propositions \ref{nwolffest} and \ref{wolffest}) such that $u\geq0$ and 

\bal
C^{-1}(N,p,s,\xL_K)W_{s,p}^{\frac{d(x)}{8}}[\xm](x)\leq u(x)\leq C(N,p,s,\xL_K) W_{s,p}^{2\diam(\xO)}[\xm](x)\quad \text{a.e. in}\;\;\xO,
\eal
where $d(x)=\dist(x,\partial\xO).$ The lower estimate in the above display can be obtained as a consequence of \cite[estimate (1.25)]{KMS}. The upper estimate in the above display is an application of \cite[Theorem 5.3]{KLL} and \eqref{weakest}.

Using the above properties of the very weak solutions and the fact that if $u,g$ satisfies \eqref{weakest} and \eqref{cond} respectively then $g(u)\in L^1(\xO),$ we obtain the following result.
\begin{theorem}\label{subcritical}
Let $s\in(0,1),$ $1<p<\frac{N}{s},$ $\xm\in \mathfrak{M}_b(\xO).$ We assume that $g\in C(\BBR)$ is a nondecreasing function satisfying $g(0)=0$ and

\ba\label{cond}
\int_1^\infty (g(s)-g(-s))s^{-\frac{N(p-1)}{N-sp}-1}{\dd} s<\infty.
\ea
Then there exist a very weak solution $u$ to problem ($P_+$) satisfying \eqref{weakest} and 

\ba\label{43}
- C(N,p,s,\xL_K)W_{s,p}^{2\diam(\xO)}[\xm^-]\leq u\leq  C(N,p,s,\xL_K) W_{s,p}^{2\diam(\xO)}[\xm^+]\qquad\text{a.e. in}\;\;\xO.
\ea

In addition, for any $q\in(0,\frac{N(p-1)}{N-s})$ and $h\in(0,s),$ there exists a positive constant $c=c(N,p,s,\xL_K,q,h,|\xO|)$ such that

\ba\label{est2nonlinearintro}\BAL
\left(\int_\xO |g(u)|{\dd} x\right)^\frac{1}{p-1}&+\left(\int_{\BBR^N}\int_{\BBR^N}\frac{|u(x)-u(y)|^q}{|x-y|^{N+hq}}{\dd} x{\dd} y\right)^\frac{1}{q}\leq c(|\xm|(\xO))^\frac{1}{p-1}.
\EAL
\ea

\end{theorem}
\noindent We note here that the integral conditions \eqref{cond} and \eqref{condp=2} coincide for $p=2$.

Let us consider problem $(P_+)$ with a power absorption, i.e.,
\ba\label{absorptionpower}
\left\{
\BAL
L u +|u|^{\xk-1}u&=\xm\quad&&\text{in}\;\; \xO\\
u&=0\quad &&\text{in}\;\; \BBR^N\setminus\xO.
\EAL\right.
\ea
We first notice that the function $g(t)=|t|^{\xk-1}t$ with $k>0$ satisfies \eqref{cond} if and only if $0<\xk<\frac{N(p-1)}{N-s},$ hence problem \eqref{absorptionpower} admits a very weak solution in this case. In the supercritical case $\xk\geq\frac{N(p-1)}{N-s},$ the sufficient condition for the solvability of problem \eqref{absorptionpower} is expressed in terms of the Bessel capacity $\mathrm{Cap}_{{sp,\frac{\xk}{\xk-p+1}}}$ as follows.

\begin{theorem}\label{powerabsorptionsupercritical}
Let $s\in(0,1),$ $1<p<\frac{N}{s},$ $\xk>p-1$ and $\xm\in \mathfrak{M}_b(\xO).$ In addition we assume that $\xm$ is absolutely continuous with respect to the Bessel capacity $\mathrm{Cap}_{{sp,\frac{\xk}{\xk-p+1}}}$. Then there exists a very weak solution $u$ to problem \eqref{absorptionpower} such that

\ba\label{41}
-C W_{s,p}^{2\diam(\xO)}[\xm^-]\leq u\leq C W_{s,p}^{2\diam(\xO)}[\xm^+]\quad\text{a.e. in}\;\;\xO.
\ea
In addition, for any $q\in(0,\frac{N(p-1)}{N-s})$ and $h\in(0,s),$ there exists a positive constant $c=c(N,p,s,\xL_K,q,h,|\xO|)$ such that

\ba\label{est2nonlinearpower}\BAL
\left(\int_\xO |u|^\xk{\dd} x\right)^\frac{1}{p-1}+\left(\int_{\BBR^N}\int_{\BBR^N}\frac{|u(x)-u(y)|^q}{|x-y|^{N+hq}}{\dd} x{\dd} y\right)^\frac{1}{q}\leq c(|\xm|(\xO))^\frac{1}{p-1}.
\EAL
\ea
\end{theorem}

In view of the discussion on the existence of solutions to problem \eqref{sourcpowerintro}, we expect that the existence phenomenon occurs for ($P_-$) only for measures $\xm\in \mathfrak{M}_b(\xO)$ with small enough total mass. Indeed, using the Schauder fixed point theorem and sharp weak Lebesgue estimates, we prove the following existence result for any $\xm\in \mathfrak{M}_b(\xO)$ with small enough total mass.

\begin{theorem}\label{sourcegeneral} Let $s\in(0,1),$ $1<p<\frac{N}{s}$ and $\tau\in\mathfrak{M}_b(\xO)$ be such that  $|\tau|(\xO)\leq 1.$ Assume that $ g \in C(\R)$ is a nondecreasing function satisfying \eqref{cond} and
\ba\label{condpower}
 |g(s)|\leq a|s|^d \quad \text{for some } a>0,\;d>1\;\; \text{ and for any } |s|\leq 1.
 \ea

Then there exists a positive constant $\xr_0$ depending on $N,|\Gw|,\xL_g,\Gl_K,a,s,p,d,|\xO|$ such that for every $\xr \in (0,\xr_0)$ the following problem
	\ba\label{trans0} \left\{
\BAL
Lv&=g(v)+\xr\tau\quad&&\text{in}\;\; \xO\\
v&=0 \quad &&\text{in}\;\; \BBR^N\setminus\xO,
\EAL
\right.
\ea
admits a very weak solution $v$ satisfying
	\ba\label{estM2}
\||v|^{p-1}\|_{L_w^{\frac{N}{N-sp}}(\BBR^N)}^* \leq t_0.
	\ea
Here, $t_0>0$ depends on $N,|\Gw|,\xL_g,\Gl_K,a,s,p,d,\xr_0.$ In addition, for any $q\in(0,\frac{N(p-1)}{N-s})$ and $h\in(0,s),$ there exists a positive constant $c$ depending only on $N,p,s,\xL_g,\xL_K,q,h,|\xO|,a,d,\xr_0$ and $t_0,$ such that

\ba\label{est2nonlinearsource}\BAL
\left(\int_{\BBR^N}\int_{\BBR^N}\frac{|v(x)-v(y)|^q}{|x-y|^{N+hq}}{\dd} x{\dd} y\right)^\frac{1}{q}\leq c(1+\xr|\tau|(\xO))^\frac{1}{p-1}.
\EAL
\ea
\end{theorem}
\noindent In the linear case, i.e. $p=2$, problem ($P_-$) with $L=(-\xD)^s$ was thoroughly studied in \cite{CFV}. More precisely, the authors in \cite{CFV} showed that the same existence result occurs provided $g$ satisfies \eqref{condp=2} and \eqref{condpower}.

Problem ($P_-$) with $g(t)=|t|^{\xk-1}t$ and $\xm\in\mathfrak{M}_b^+(\xO)$ becomes 

\ba\label{powersource}\left\{
\BAL
L v&=|v|^{\xk-1}v+\xr\tau\quad&&\text{in}\;\; \xO\\
v&=0\quad\quad &&\text{in}\;\; \BBR^N\setminus\xO.
\EAL
\right.
\ea
When $p=2$, problem ($P_-$) with $L=(-\xD)^s$ and $\tau=\xd_0$ was studied in \cite{CQ}. Here $\xd_0$ denotes the dirac measure concentrated at a point $x_0\in\xO.$ In particular, the authors in \cite{CQ} established that if $\xk\geq\frac{N(p-1)}{N-sp}$ and $u$ is a nonnegative solution of  \eqref{powersource} then $\xr=0.$
Concerning problem \eqref{powersource}, conditions \eqref{cond} and \eqref{condpower} are satisfied if $\xk$ belongs to the subcritical range, that is when $p-1<\xk<\frac{N(p-1)}{N-p}.$ In general, a sufficient condition for the solvability of \eqref{powersource} is the following.

\begin{proposition}\label{exissource} Let $s\in(0,1),$ $1<p<\frac{N}{s},$ $\xk>p-1$ and $\tau\in\mathfrak{M}_b^+(\xO) $ be such that

\ba\label{subcriticalcondition}
W_{s,p}^{2\diam(\xO)}[(W_{s,p}^{2\diam(\xO)}[\tau])^\xk]\leq M W_{s,p}^{2\diam(\xO)}[\tau]\;\;a.e.\;\;\text{in}\;\;\xO,
\ea
for some positive constant $M.$ Then problem \eqref{powersource} admits a nonnegative very weak solution $u$ for some $\xr>0.$ Furthermore, there holds

\ba\label{estsourcewolff}
M^{-1} W_{s,p}^{\frac{d(x)}{8}}[\xm](x)\leq u(x)\leq M W_{s,p}^{2\diam(\xO)}[\xr\tau](x)\qquad\text{for a.e.}\; x\in\xO,
\ea
where ${\dd}\xm=u^\xk {\dd} x+\xr \dd\tau$ and the positive constant $M$ depends only on $C,N,p,q,\xL_K.$ 
\end{proposition}

Finally, inspired from Phuc and Verbitsky's ideas in \cite{PVsource} and \cite{PV}, we establish the following existence result in the whole range $\xk>p-1.$

\begin{theorem}\label{supercrpower1}
Let $s\in(0,1),$ $1<p<\frac{N}{s},$ $\xk>p-1$ and  $\tau\in \mathfrak{M}_b^+(\xO)$ with compact support in $\xO.$ Then the following statements are equivalent.

(i) Problem \eqref{powersource} admits a nonnegative very weak solution $u_\xr$ for some $\xr>0$ such that

\ba\label{estsourcewolff3}
C_1^{-1}W_{s,p}^{\frac{d(x)}{8}}[\xm](x)\leq u_\xr(x)\leq C_1 W_{s,p}^{2\diam(\xO)}[\xr\tau](x)\qquad\text{for a.e.}\; x\in\xO,
\ea
where ${\dd}\xm=u^\xk {\dd} x+\xr \dd\tau$ and for some constant $C_1>0.$

(ii) There exists a positive constant $C_2$ such that

\ba\label{caocon2}
\tau(E)\leq C_2 \mathrm{Cap}_{{sp,\frac{\xk}{\xk-p+1}}}(E)
\ea

for any Borel set $E\subset\BBR^N.$

(iii) There exists a positive constant $C_3$ such that
\ba\label{57}
\int_{B}(W_{s,p}^{2\diam(\xO)}[\tau_{\lfloor B}])^\frac{\xk}{p-1}{\dd} x \leq C_3\tau(B)
\ea
for any ball $B\subset\BBR^N,$ where ${\dd}\tau_{\lfloor B}=\chi_B{\dd} \tau.$

(iv) There exists a positive constant $C_4$ such that
\bal
W_{s,p}^{2\diam(\xO)}[(W_{s,p}^{2\diam(\xO)}[\tau])^\xk]\leq C_4W_{s,p}^{2\diam(\xO)}[\tau]\quad\text{a.e.}\;\;\text{in}\;\;\xO.
\eal
\end{theorem}
\noindent We note here that if {\small$p-1<q<\frac{N(p-1)}{N-sp}$\normalsize} then {\small$\frac{spq}{q-p+1}>N$\normalsize}, this implies that {\small$\mathrm{Cap}_{{sp,\frac{q}{q-p+1}}}(\{x\})>0$\normalsize} for any $x\in\BBR^N$ (see \cite[Section 2.6]{Adbook}). Hence, the statement (ii) in the above theorem is always satisfied in the subcritical range.

\medskip

\noindent \textbf{Organization of the paper.} Section \ref{secvery} is devoted to the study of the very weak solutions to problem \eqref{absorption} with $\tilde g\equiv0.$ In Section \ref{secabs}, we discuss problem ($P_+$) as well as Theorems \ref{subcritical} and \ref{powerabsorptionsupercritical} are proved in Subsections \ref{secsubcr} and \ref{secpower} respectively. In section \ref{secsource}, we deal with problem ($P_-$). More precisely, we prove Theorem \ref{sourcegeneral} in Subsection \ref{secsourcesubcr} and demonstrate Proposition \ref{exissource} and Theorem \ref{supercrpower1} in Subsection \ref{secsourcepower}.

\medskip
\noindent \textbf{Acknowledgement.} The author wish to thank Professor L. V\'eron for useful discussions. The research project was supported by the Hellenic Foundation for Research and Innovation (H.F.R.I.) under the  “2nd Call for H.F.R.I. Research Projects to support Post-Doctoral Researchers” (Project Number: 59).

\section{Very weak solutions}\label{secvery}
We start with the definition of the fractional spaces, which will be used frequently in this work. For any $s\in(0,1)$ and $q>0,$ we denote by $W^{s,q}(\BBR^N)$ the fractional space 
 
 \ba
W^{s,q}(\BBR^N):=\left\{\int_{\BBR^N}\int_{\BBR^N}\frac{|u(x)-u(y)|^q}{|x-y|^{N+sq}}{\dd} x{\dd} y+\int_{\BBR^N}|u|^q{\dd} x<\infty\right\},\label{sobolev}
\ea
endowed with the quasinorm
\bal
\norm{u}_{W^{s,q}(\BBR^N)}:=\left(\int_{\BBR^N}\int_{\BBR^N}\frac{|u(x)-u(y)|^q}{|x-y|^{N+sq}}{\dd} x{\dd} y\right)^{\frac{1}{q}}
+\left(\int_{\BBR^N}|u|^q{\dd} x\right)^{\frac{1}{q}}.
\eal
When $q\geq 1,$  $W^{s,q}(\BBR^N)$ is a Banach space and is called fractional Sobolev space. Finally, for any $p\in (1,\frac{N}{s}),$ we denote by $W_0^{s,p}(\xO)$ the closure of $C_0^\infty(\xO)$ in the norm $\norm{\cdot}_{W^{s,p}(\BBR^N)}$ and by $(W^{s,p}_0(\xO))^*$ its dual space.

\subsection{Weak solutions and a priori estimates} In this subsection, we introduce the notion of the weak solution of the following problem
\ba\label{weaksol}
\left\{
\BAL
L u &=\xm\quad&&\text{in}\;\; \xO\\
u&=0\quad &&\text{in}\;\; \BBR^N\setminus\xO,
\EAL\right.
\ea
where $\xm\in (W^{s,p}_0(\xO))^*.$ In addition, when $\xm\in L^{p'}(\xO),$ we establish a priori estimates, which will be used in the construction of the very weak solutions of the above problem with measure data.

\begin{definition}
Let $s\in(0,1),$ $1<p<\frac{N}{s}$ and $\xm\in (W^{s,p}_0(\xO))^*.$ We will say that $u\in W^{s,p}_0(\xO)$ is a weak solution of \eqref{weaksol},
if it satisfies
\bal
\int_{\BBR^N}\int_{\BBR^N}|u(x)-u(y)|^{p-2}(u(x)-u(y))(\xf(x)-\xf(y))K(x,y){\dd} x{\dd} y=<\xm,\xf>\quad\forall\xf\in W^{s,p}_0(\xO).
\eal
\end{definition}

Let us now give the definition of weak supersolutions of $L$ in $\xO.$
\begin{definition}
Let $s\in(0,1)$ and $1<p<\frac{N}{s}.$ We will say that $u\in W^{s,p}(\BBR^N)$ is a weak supersolution (resp. subsolution) of $L$ in $\xO,$
if and only if satisfies
\bal
\int_{\BBR^N}\int_{\BBR^N}|u(x)-u(y)|^{p-2}(u(x)-u(y))(\xf(x)-\xf(y))K(x,y){\dd} x{\dd} y\geq0\;(resp. \leq0)
\eal
for any nonnegative $\xf\in W^{s,p}_0(\xO).$
\end{definition}

Next we state the comparison principle.

\begin{proposition}[{\cite[Lemma 6]{KKP}}]\label{comparison}
Let $u\in W^{s,p}(\BBR^N)$ be a weak supersolution of $L$ in $\xO$ as well as let  $v\in  W^{s,p}(\BBR^N)$ be a weak subsolution of $L$ in $\xO$ such that $(v-u)_+\in W^{s,p}_0(\xO).$ Then, $u\geq v$ a.e. in $\BBR^N.$
\end{proposition}

In view of the proof \cite[Theorem 2.3]{diCKP_local}, we may obtain the following existence result.

\begin{proposition}\label{existenceweaksol}
For any $\xm\in (W^{s,p}_0(\xO))^*$ there exists a unique weak solution of \eqref{weaksol}.
\end{proposition}

In order to state the first a priori estimate for the weak solution of \eqref{weaksol}, we need to give the definition and the main properties of Marcinkiewicz spaces. Let $D \subset \R^N$ be a domain. Denote  $L^p_w(D)$, $1 \leq p < \infty,$ the
weak $L^p$ space (or Marcinkiewicz space) defined as follows. A measurable function $f$ in $D$
belongs to this space if there exists a constant $c$ such that
\bel{distri} \gl_f(a):=|\{x \in D: |f(x)|>a\}|\leq ca^{-p},
\forevery a>0. \ee
The function $\gl_f$ is called the distribution function of $f$. For $p \geq 1$, denote
\bal
 L^p_w(D)=\{ f \text{ Borel measurable}:
\sup_{a>0}a^p\gl_f(a)<\infty\},
\eal
\ba\label{semi}
\norm{f}^*_{L^p_w(D)}=(\sup_{a>0}a^p\gl_f(a))^{\frac{1}{p}}. \ea
The $\norm{.}_{L^p_w(D)}^*$ is not a norm, but for $p>1$, it is
equivalent to the norm
\ba\label{normLw} \norm{f}_{L^p_w(D)}=\sup\left\{
\frac{\int_{\gw}|f|{\dd} x}{|\gw|^{1/p'}}:\gw \sbs D, \gw \text{
	measurable},\, 0<|\gw|<\infty \right\}. \ea
More precisely,
\bel{equinorm} \norm{f}^*_{L^p_w(D)} \leq \norm{f}_{L^p_w(D)}
\leq \myfrac{p}{p-1}\norm{f}^*_{L^p_w(D)}. \ee
Notice that,
\bal
L_w^p(D) \sbs L^{r}(D) \forevery r \in [1,p).
\eal
From \eqref{semi} and \eqref{equinorm}, one can derive the following estimate which is useful in the sequel.
\ba \label{ue}
\int_{\{|u| \geq s\} }{\dd} x \leq s^{-p}\norm{u}^p_{L_w^p(D)}.
\ea

\begin{proposition}\label{weaklp}
Let $\xm\in L^{p'}(\xO)$ and $u\in W^{s,p}_0(\xO)$ be the unique weak solution of \eqref{weaksol}. Then there exists a positive constant $C=C(p,s,N,\xL_K)$ such that

\ba\label{est7}
\norm{|u|^{p-1}}^{*}_{L^{\frac{N}{N-sp}}_w(\BBR^N)}\leq C \int_\xO|\xm|{\dd} x.
\ea

\end{proposition}
\begin{proof}

Let $k>0.$ Taking $T_k(u)$ as test function and using the fact that 

\bal
|u(x)-u(y)|^{p-2}(u(x)-u(y))(T_k(u)(x)-T_k(u)(y))\geq |T_k(u)(x)-T_k(u)(y)|^p\quad\forall x,y\in \BBR^N,
\eal
we obtain

\ba\label{trunc}
\int_{\BBR^N}\int_{\BBR^N}\frac{|T_k(u)(x)-T_k(u)(y)|^p}{|x-y|^{N+sp}}{\dd} x{\dd} y\leq \xL_K k \int_{\xO}|\xm| {\dd} x.
\ea
Now, by the above inequality and the fractional Sobolev inequality we have

\bal
|\{|u(x)|\geq k\}|= |\{|T_k(u)(x)|\geq k\}|\leq k^{-\frac{Np}{N-sp}}\int_{\BBR^N}|T_{k}(u)(x)|^{\frac{Np}{N-sp}}{\dd} x\leq Ck^{-\frac{N(p-1)}{N-sp}}\left(\int_\xO|\xm|{\dd} x\right)^{\frac{N}{N-sp}},
\eal
which implies the desired result.
\end{proof}

\begin{proposition}
Let $\xm\in L^{p'}(\xO)$ and $u\in W^{s,p}_0(\BBR^N)$ be the unique weak solution of \eqref{weaksol}. Then there exists a positive constant $C=C(p,s,N,\xL_K)$ such that

\ba\label{est5}
\int_{\BBR^N}\int_{\BBR^N}\frac{|u(x)-u(y)|^p}{(d+|u(x)|+|u(y)|)^\xi}\frac{{\dd} x{\dd} y}{|x-y|^{N+sp}}\leq \frac{Cd^{1-\xi}}{(\xi-1)} \int_{\xO}|\xm| {\dd} x
\ea
for any $\xi>1$ and $d>0.$
\end{proposition}
\begin{proof}
The proof is very similar to that of \cite[Lemma 3.1]{KMS} (see also \cite[Lemma 8.4.1]{KMS2}). For the sake of convenience we give it below.

Set $\xf_\pm:=\pm(d^{1-\xi}-(d+u_\pm)^{1-\xi}).$ Using $\xf_\pm$ as test function we obtain

\bal
\int_{\BBR^N}\int_{\BBR^N}|u(x)-u(y)|^{p-2}(u(x)-u(y))(\xf_\pm(x)-\xf_\pm(y))K(x,y){\dd} x{\dd} y=\int_\xO \xf_\pm\xm {\dd} x.
\eal
Now,

\bal
(\xf_\pm(x)-\xf_\pm(y))=\pm(\xi-1)(u_\pm(x)-u_\pm(y))\int_0^1(d+t u_\pm(y)+(1-t)u_\pm(x))^{-\xi} dt,
\eal
which implies

\bal
&|u(x)-u(y)|^{p-2}(u(x)-u(y))(\xf_\pm(x)-\xf_\pm(y))K(x,y)\\
&\quad\quad\geq (\xi-1)|u(x)-u(y)|^{p-2}(u_\pm(x)-u_\pm(y))^2(d+|u(y)|+|u(x)|)^{-\xi}.
\eal

Combining all above we can easily reach the desired result.

\end{proof}

We conclude this subsection by the following a priori estimate for the weak solutions of \eqref{weaksol} in the whole range $1<p<\frac{N}{s}.$

\begin{proposition}\label{lpest}
Let $\xm\in L^{p'}(\xO)$ and $u\in W^{s,p}_0(\BBR^N)$ be the unique weak solution of \eqref{weaksol}.  For any $q\in(0,\frac{N(p-1)}{N-s})$ and $h\in(0,s),$ there exists a positive constant $c$ depending only on $N,s,p,\xL_K,q$ and $|\xO|$ such that

\ba\label{est1}
\left(\int_{\BBR^N}\int_{\BBR^N}\frac{|u(x)-u(y)|^q}{|x-y|^{N+hq}}{\dd} x{\dd} y\right)^\frac{1}{q}
\leq c\bigg(\int_\xO|\xm|{\dd} x\bigg)^\frac{1}{p-1}.
\ea
\end{proposition}
\begin{proof}
The proof is an adaptation of the argument in \cite[Lemma 3.2]{KMS}.
Let $R=\diam(\xO)$ and $x_0\in\xO.$ Taking into account that $u=0$ a.e. in $\BBR^N\setminus\xO,$ we can easily prove that

\ba\label{nest6}\BAL
\int_{\BBR^N}\int_{\BBR^N}\frac{|u(x)-u(y)|^q}{|x-y|^{N+hq}}{\dd} x{\dd} y&\approx \int_{B_{2R}(x_0)}\int_{B_{2R}(x_0)}\frac{|u(x)-u(y)|^q}{|x-y|^{N+hq}}{\dd} x{\dd} y+\int_\xO |u|^q {\dd} x\\
&\approx  \int_{B_{2R}(x_0)}\int_{B_{2R}(x_0)}\frac{|u(x)-u(y)|^q}{|x-y|^{N+hq}}{\dd} x{\dd} y.
\EAL
\ea 

Now, by H\"{older} inequality we obtain

\ba\label{est6}\BAL
&\int_{B_{2R}(x_0)}\int_{B_{2R}(x_0)}\frac{|u(x)-u(y)|^q}{|x-y|^{N+hq}}{\dd} x{\dd} y\\
&=\int_{B_{2R}(x_0)}\int_{B_{2R}(x_0)}\left(\frac{|u(x)-u(y)|^p}{(d+|u(x)|+|u(y)|)^\xi|x-y|^{ps}}(d+|u(x)|+|u(y)|)^\xi|x-y|^{p(s-h)}\right)^\frac{q}{p}
\frac{{\dd} x{\dd} y}{|x-y|^N}\\
&\leq\left(\int_{B_{2R}(x_0)}\int_{B_{2R}(x_0)}\frac{|u(x)-u(y)|^p}{(d+|u(x)|+|u(y)|)^\xi|x-y|^{N+sp}}{\dd} x{\dd} y\right)^\frac{q}{p}\\
&\quad\quad\cdot \left(\int_{B_{2R}(x_0)}\int_{B_{2R}(x_0)}\frac{(d+|u(x)|+|u(y)|)^{\frac{\xi q}{p-q}}}{|x-y|^{N-\frac{qp(s-h)}{p-q}}}{\dd} x{\dd} y\right)^\frac{p-q}{p}.
\EAL\ea
Setting 
\bal
d=\left(\int_\xO|u(y)|^{\frac{\xi q}{p-q}} {\dd} x\right)^\frac{p-q}{\xi q}
\eal
and combining \eqref{est5} and \eqref{est6}, we conclude 

\ba\label{d}
\int_{B_{2R}(x_0)}\int_{B_{2R}(x_0)}\frac{|u(x)-u(y)|^q}{|x-y|^{N+hq}}{\dd} x{\dd} y\leq c d^{\frac{q}{p}}\bigg(\int_\xO|\xm|{\dd} x\bigg)^\frac{q}{p}.
\ea
Since $q<\frac{N(p-1)}{N-s},$ we may choose  $\xi>1$ such that $1<\xg:=\frac{\xi q}{(p-1)(p-q)}<\frac{N}{N-sp}.$ Hence, by \eqref{equinorm} and \eqref{est7}, we deduce

\bal
\left(\int_\xO |u|^{\xg(p-1)}\right)^\frac{1}{\xg}\leq C(\xg,N,p,s,|\xO|,\xL_K)\int_\xO|\xm|{\dd} x,
\eal
which in turn implies 
\bal
d\leq C(\xg,N,p,s,|\xO|,\xL_K)\left(\int_\xO|\xm|{\dd} x\right)^{\frac{1}{p-1}}.
\eal
The desired result follows by \eqref{nest6}, \eqref{d} and the above inequality. 
\end{proof}

\subsection{Existence and main properties}

In this subsection, we construct a very weak solution to problem \eqref{weaksol} which possesses several important properties, such as it satisfies
pointwise estimates in terms of Wolff’s potential. These estimates play an important role in the study of problems \eqref{main}.

We start with the following existence result.
\begin{proposition}\label{sola}
Let $\xm\in \mathfrak{M}_b(\xO).$ Then there exists a very weak solution to \eqref{weaksol} satisfying 

\ba\label{estn1}
\norm{|u|^{p-1}}^{*}_{L^{\frac{N}{N-sp}}_w(\BBR^N)}\leq C_1(N,p,s,\xL_K)\xm(\xO)
\ea
and

\ba\label{estn2}
\int_{\BBR^N}\int_{\BBR^N}\frac{|T_k(u)(x)-T_k(u)(y)|^p}{|x-y|^{N+sp}}{\dd} x{\dd} y\leq k \xL_K|\xm|(\xO)\quad\forall k>0.
\ea

In addition, for any $q\in(0,\frac{N(p-1)}{N-s})$ and $h\in(0,s),$ there exist a positive constant $C_2=C_2(N,p,s,\xL_K, q,h,|\xO|)$ such that 
\ba\label{est2}
\left(\int_{\BBR^N}\int_{\BBR^N}\frac{|u(x)-u(y)|^q}{|x-y|^{N+hq}}{\dd} x{\dd} y\right)^\frac{1}{q}
\leq C_2|\xm|(\xO)^\frac{1}{p-1}.
\ea

\end{proposition}
\begin{proof}
Let $\{\xr_n\}_n$ be a sequence of mollifiers and $\xm_n=\xr_n*\xm.$ Then $\xm_n\in C_0^\infty(\BBR^N)$  and $\xm_n\rightharpoonup \xm$ weakly in $\mathfrak{\BBR}^N.$ We denote by $u_n$ the weak solution of \eqref{weaksol} with $\xm=\xm_n.$ 

Let $q\in (p-1,\frac{N(p-1)}{N-s})$ and $h\in (0,s).$ Then by \eqref{est7}, \eqref{trunc} and \eqref{est1}, there exist positive constants $C_1$ and $C_2$ such that

\ba\label{n1}
\norm{|u_n|^{p-1}}^{*}_{L^{\frac{N}{N-sp}}_w(\BBR^N)}\leq C_1(N,p,s,\xL_K)\xm(\xO)\qquad \forall n\in\BBN,
\ea

\ba\label{n2}
\int_{\BBR^N}\int_{\BBR^N}\frac{|T_k(u_n)(x)-T_k(u_n)(y)|^p}{|x-y|^{N+sp}}{\dd} x{\dd} y\leq k \xL_K\xm(\xO) \qquad\forall k>0\;\;\text{and}\;\; n\in\BBN
\ea
and

\ba\label{n3}\BAL
\left(\int_{\BBR^N}\int_{\BBR^N}\frac{|u_n(x)-u_n(y)|^q}{|x-y|^{N+hq}}{\dd} x{\dd} y\right)^\frac{1}{q}\leq C_2(N,p,s,\xL_K, q,h)\xm(\xO)^\frac{1}{p-1}\qquad\forall n\in\BBN.
\EAL
\ea
As in steps 2-3 of the proof of \cite[Theorem 3.4]{MMOP} (see also \cite[Step 2-3 pg 275]{Vbook}), by \eqref{n1}-\eqref{n3} and the fractional Sobolev embedding theorem (see e.g. [14, Corollary 7.2]), we may deduce the existence of a subsequence (still denoted by $\{u_n\}$) and a function $u\in W^{h,q}(\BBR^N)$ satisfying the following properties:

(i) $u_n\to u$ a.e. in $\BBR^N,$ $u=0$ a.e. in $\BBR^N\setminus\xO$ and $\norm{u-u_n}_{W^{h,q}(\BBR^N)}\to 0.$

(ii) $T_k(u)\in W_0^{s,p}(\xO)$ for any $k>0.$

(iii) $u\in W^{\xs,r}(\BBR^N)$ for any $0<r<\frac{N(p-1)}{N-s}$ and for any $0<\xs<s.$ 

\noindent Combining all above we obtain the desired result.
\end{proof}

In the next theorem, we establish a priori pointwise estimates for nonnegative very weak solutions of problem \eqref{weaksol} with $\xm\in \mathfrak{M}_b^+(\xO).$

\begin{proposition}\label{nwolffest}
Let $\xm\in \mathfrak{M}_b^+(\xO)$ and $u$ be a nonnegative very weak solution of \eqref{weaksol}. Then there exists a positive constant $C$ depending only on $N,s,p,\xL_K$ such that
\ba\label{nest3}\BAL
& C^{-1}W_{s,p}^{\frac{d(x)}{8}}[\xm](x)\leq u(x)\\
&\;\;\leq C\bigg(\essinf_{B_{\frac{d(x)}{4}}(x)} u +W_{s,p}^{\frac{d(x)}{2}}[\xm](x)+ \bigg(\bigg(\frac{d(x)}{4}\bigg)^{sp}\int_{\BBR^N\setminus B_{\frac{d(x)}{4}}(x)}\frac{u(y)^{p-1}}{|x-y|^{N+sp}}{\dd} y\bigg)^\frac{1}{p-1}\bigg)
\EAL
\ea
for a.e. $x\in \xO.$
\end{proposition}
\begin{proof}
In view of the proof of \cite[Lemma 7]{KKP}, we have that $u_k=\min(u,k)$ is a nonnegative weak supersolution. Hence, there exists a nonnegative Radon measure $\xm_k\in \mathfrak{M}^+(\xO)$ such that 

\ba\label{n4}
\int_{\BBR^N}\int_{\BBR^N}|u_k(x)-u_k(y)|^{p-2}(u_k(x)-u_k(y))(\xf(x)-\xf(y))K(x,y){\dd} x{\dd} y=\int_\xO \xf(x){\dd} \xm_k
\ea
for any $\xf\in C_0^\infty(\xO).$ Since $u_k\to u$ in $\BBR^N,$ we have that $\norm{u-u_k}_{W^{h,q}(\BBR^N)}\to 0$ for any $h\in (0,s)$ and $q\in(0,\frac{N(p-1)}{N-s}).$ This, together with \eqref{n4}, implies 

\ba
\int_\xO \xf(x){\dd} \xm_k\to \int_\xO \xf(x){\dd} \xm\qquad\forall \xf\in C_0^\infty(\xO).
\ea 
Now, we remark that, in view of the proof of \cite[Theorem 1.3]{KMS}, we may apply \cite[estimate (1.25)]{KMS} to $u_k.$ Hence, 
\bal
C^{-1}W_{s,p}^{\frac{d(x)}{8}}[\xm_k](x)\leq u_k(x)\qquad \text{for a.e.}\;x\in\xO\;\;\text{and}\;\;\forall k>0.
\eal
Letting $k\to\infty$ in the above inequality and using some elementary manipulations, we may obtain the lower estimate in \eqref{nest3}.

For the upper estimate in \eqref{nest3}, by \cite[Theorem 9]{KKP}, we have that 
\bal
v_k(x):=\essliminf_{y\to x}u_k(y)=u_k(x)\qquad \text{for a.e.} \;x\in\BBR^N.
\eal
Hence, $v_k$ is a lower semicontinuous functions in $\xO$ and a nonnegative weak supersolution. By \cite[Theorem 12]{KKP}, $v_k$ is $(s,p)-$superharmonic function in $\xO$ (see \cite[Definition 1]{KKP} for the definition of $(s,p)-$superharmonic function). This, together with \cite[Lemma 12]{KKP}, implies that $v:=\lim_{k\to\infty} v_k$ is $(s,p)-$superharmonic function in $\xO$ and $v=u$  a.e. in $\BBR^N.$ The desired result follows by applying \cite[Theorem 5.3]{KLL} to $v$ and the fact that $v=u$ a.e. in $\BBR^N.$
\end{proof}

\begin{proposition}\label{wolffest}
Let $\xm\in \mathfrak{M}_b(\xO).$ Then there exists a very weak solution $u$ of \eqref{weaksol} and a positive constant $C$ depending only on $N,s,p$  and $\xL_K$ such that
\ba\label{est3}
-C W_{s,p}^{2\diam(\xO)}[\xm^-]\leq u\leq CW_{s,p}^{2\diam(\xO)}[\xm^+]\qquad \text{a.e. in}\;\;\xO.
\ea
\end{proposition}
\begin{proof}
 Let $u$ be the solution constructed in Proposition \ref{sola} and $x_0\in\xO.$ Set $R=\diam(\xO),$ $\xm_n=\xr_n*\xm$ and $\xm_n^\oplus=\xr_n*\xm^+.$ We denote by  $v_n^{\oplus}\in W^{s,p}_0(\xO)$ the solution of
\bal\left\{\BAL
Lv_n^\oplus&=\xm_n^\oplus&&\quad\text{in}\;B_{2R}(x_0)\\
v_n^\oplus&=0&&\quad\text{in}\;\BBR^N\setminus B_{2R}(x_0).
\EAL\right.
\eal

By Proposition \ref{comparison}, we have that $v_n^{\oplus}\geq 0$ and $v_n^{\oplus}\geq u_n,$
where $u_n\in W^{s,p}_0(\xO)$ is the weak solution of \eqref{weaksol} with $\xm=\xm_n.$  By statements (i)-(iii) in the proof of Proposition \ref{sola}, there exist subsequences $\{u_{n_k},v_{n_k}^{\oplus}\}_{k=1}^\infty$  such that $u_{n_k}\to u$ and $v_{n_k}^\oplus\to v^\oplus$ a.e. in $\BBR^N$ and 
\bal
\norm{u-u_{n_k}}_{W^{h,q}(\BBR^N)} +\norm{v^{\oplus}-v_{n_k}^{\oplus}}_{W^{h,q}(\BBR^N)}\to 0
\eal
for any $h\in (0,s)$ and $q\in(0,\frac{N(p-1)}{N-s}).$ Combining all above, we may deduce that $u\leq v^\oplus$ a.e. in $\BBR^N$ and $v^\oplus$ is a nonnegative very weak solution to

\bal\left\{\BAL
Lv^\oplus&=\xm^+&&\quad\text{in}\;B_{2R}(x_0)\\
v^\oplus&=0&&\quad\text{in}\;\BBR^N\setminus B_{2R}(x_0).
\EAL\right.
\eal
In addition, by Proposition \ref{wolffest}, there exists a positive constant $C=C(p,s,\xL_K,N)$ such that
\ba\label{6}
u(x)\leq v^\oplus(x)\leq C\bigg(W_{s,p}^{R}[\xm^+](x)+\essinf_{B_\frac{R}{2}(x)} v^\oplus
+\mathrm{Tail}{(v^\oplus;x,\frac{R}{2})}\bigg)\;\;\text{for a.e.}\; x\in\xO,
\ea
where

\bal
\mathrm{Tail}{(v^\oplus;x,\frac{R}{2})}=\bigg(\bigg(\frac{R}{2}\bigg)^{sp}\int_{\BBR^N\setminus B_{\frac{R}{2}}(x)}\frac{|v^\oplus(y)|^{p-1}}{|x-y|^{N+sp}}{\dd} y\bigg)^\frac{1}{p-1}.
\eal
By \eqref{estn1} and \eqref{equinorm}, we derive that

\ba\label{7}
\essinf_{B_\frac{R}{2}(x)} v^\oplus\lesssim \bigg(\dashint_{B_{\frac{R}{2}}(x)}|v^\oplus|^{p-1}{\dd} x\bigg)^\frac{1}{p-1}\lesssim R^{-\frac{N-sp}{p-1}}\xm^+(B_R(x))^\frac{1}{p-1}\lesssim W_{s,p}^{2R}[\xm^+](x),
\ea
and
\ba\BAL\label{8}
\mathrm{Tail}{(v^\oplus;x_0,\frac{R}{2})}\lesssim\bigg(\dashint_{B_{2R}(x)}|v^\oplus|^{p-1}{\dd} x\bigg)^\frac{1}{p-1}
\lesssim W_{s,p}^{2R}[\xm^+](x),
\EAL
\ea
where the implicit constants in \eqref{7} and \eqref{8} depend only on $p,s,\xL_K,N.$ Combining \eqref{6}-\eqref{8}, we obtain the upper bound in \eqref{est3}.

The proof of the lower bound in \eqref{est3} is similar and we omit it.
\end{proof}

\section{Nonlocal equations with absorption nonlinearities}\label{secabs}

\subsection{The variational problem}

We assume that $g\in C(\BBR)$ and $rg(r)\geq0.$ Let $\xO\subset\BBR^N$ be an open bounded domain and $\xm\in (W^{s,p}_0(\xO))^*.$ Set $G(r)=\int_0^rg(s){\dd} s,$
\bal
J(v)=\frac{1}{p}\int_{\BBR^N}\int_{\BBR^N}|v(x)-v(y)|^pK(x,y){\dd} x{\dd} y+ \int_{\xO}G(v){\dd} x-<\xm,v>
\eal
and

\bal
\mathbf{X}_G(\xO)=\{v\in W^{s,p}_0(\xO):\;G(v)\in L^1(\xO)\}.
\eal

\begin{theorem}\label{var}
Let $s\in(0,1),$ $1<p<\frac{N}{s}$ and $\xm\in (W^{s,p}_0(\xO))^*.$ Then, there exists a minimizer $u_\xm$ of $J$ in $\mathbf{X}_G(\xO).$ Furthermore, $u_\xm$ is a weak solution of $J,$ in the sense of

\ba\label{weak}
\int_{\BBR^N}\int_{\BBR^N}|u_\xm(x)-u_\xm(y)|^{p-2}(u_\xm(x)-u_\xm(y))(\xz(x)-\xz(y))K(x,y){\dd} x{\dd} y+ \int_\xO g(u_\xm)\xz {\dd} x=<\xm,\xz>
\ea
for any $\xz\in W^{s,p}_0(\xO)\cap L^\infty(\xO).$

If $g$ is nondecreasing the solution $u_\xm$ is unique and the mapping $\xm\mapsto u_\xm$ is nondecreasing.
\end{theorem}
\begin{proof}

We adapt the argument used in the proof of \cite[Theorem 5.1]{egw}.
Let $\{v_n\}$ be a minimizing sequence. Taking in to account that $G(t)\geq0$ for any $t\in \BBR$ and the fractional Sobolev inequality, we can easily show the existence of a positive constant  $C=C(p,\xO,\xL_K)$ such that
\ba\label{coe}
\norm{v_n}_{W_0^{s,p}(\xO)}^p\leq C(J(v_n)+\norm{\xm}_{(W_0^{s,p}(\xO))^*}^{p'})\qquad\forall n\in\BBN.
\ea
This implies that  $v_n$ is uniformly bounded in $W_0^{s,p}(\xO).$ Thus, by the fractional Sobolev embedding theorem (see e.g. \cite[Corollary 7.2]{hitch}) and the fact that $W_0^{s,p}(\xO)$ is a reflexive Banach space, we may prove the existence of a subsequence, still denoted by $\{v_{n}\}$ and a function  $v\in W^{s,p}_0(\xO)$ such that there hold:

\medskip

(i) $v_{n}\to v$ a.e. in $\BBR^N.$

(ii) $v_{n}\to v$ in $L^q(\xO)$ for any $q\in[1,\frac{N p}{N-sp}).$ 

(iii) $v_{n}\rightharpoonup v$ in $W_0^{s,p}(\xO)$ and $v_{n}\rightarrow v$ in $W^{h,q}(\BBR^N)$ for any $h\in(0,s)$ and $q\in(1,p).$ 

\medskip

\noindent By Fatou's Lemma, we obtain
\bal
&\frac{1}{p}\int_{\BBR^N}\int_{\BBR^N}|v(x)-v(y)|^pK(x,y){\dd} x{\dd} y+ \int_{\xO}G(v){\dd} x\\
&\qquad\qquad\leq\liminf_{k\to \infty}\frac{1}{p}\int_{\BBR^N}\int_{\BBR^N}|v_k(x)-v_k(y)|^pK(x,y){\dd} x{\dd} y+ \int_{\xO}G(v_k){\dd} x.
\eal
Hence $v$ is a minimizer. If $g$ is nondecreasing, the uniqueness of the minimizer follows by the fact that $J$ is strictly convex.

We next show \eqref{weak}. Let $v_k$ be the minimizer of $J$ associated with  {\small$g_k=\max(-k,\min(g,k)).$\normalsize} Then, in view of the proof of \cite[Theorem 2.3]{diCKP_local},  $v_k$ satisfies

\ba\label{weak2}
\int_{\BBR^N}\int_{\BBR^N}|v_k(x)-v_k(y)|^{p-2}(v_k(x)-v_k(y))(\xz(x)-\xz(y))K(x,y){\dd} x{\dd} y+ \int_\xO g_k(v_k)\xz {\dd} x=<\xm,\xz>
\ea
for any $\xz\in W^{s,p}_0(\xO).$ Taking $v_k$ as test function, we have

\bal
\int_{\BBR^N}\int_{\BBR^N}|v_k(x)-v_k(y)|^{p}K(x,y){\dd} x{\dd} y&+ \int_\xO g(v_k)v_k {\dd} x=<\xm, v_k>\\
&\leq \frac{1}{p}\norm{v_k}_{W_0^{s,p}(\xO)}^p+\frac{1}{p'}\norm{\xm}_{(W_0^{s,p}(\xO))^*}^{p'},
\eal
which implies

\ba\label{11}
\int_{\BBR^N}\int_{\BBR^N}\frac{|v_k(x)-v_k(y)|^{p}}{|x-y|^{N+sp}}{\dd} x{\dd} y+ \int_\xO g_k(v_k)v_k {\dd} x\leq C(\xL_K,p)\norm{\xm}_{(W_0^{s,p}(\xO))^*}^{p'}=: M.
\ea
By the above inequality, we may deduce  that there exists a subsequence, still denoted by $\{v_{k}\}$ and a function  $v\in W^{s,p}_0(\xO)$ such that they satisfy statements (i)-(iii).

Let $\xz\in L^\infty(\xO)$ with $\norm{\xz}_{L^\infty(\xO)}=N$ and $E\subset \xO$ be a Borel set. Then, for any $\xl>0,$ we have
\bal
\int_{E\cap\{|v_k|>\xl\}}|\xz g_k(v_k)|{\dd} x\leq\frac{1}{\xl} \int_{E\cap\{v_k>\xl\}}|\xz| |v_k g_k(v_k)|{\dd} x\leq \frac{N}{\xl}\int_{\xO} v_k g_k(v_k){\dd} x\leq \frac{MN}{\xl}.
\eal
Also,

\bal
\int_{E\cap\{|u_k|\leq\xl\}}|\xz g_k(v_k)|{\dd} x\leq |E|N\sup\{|g(t)|:|t|\leq \xl\}.
\eal
Let $\xe>0,$ $\xl=\frac{2MN}{\xe}$ and $\xd= \frac{\xe}{2N\sup\{|g(t)|:|t|\leq\frac{2MN}{\xe}\}+1}$. Then for any Borel set $E\subset\xO$ with
 $|E|<\xd$, we have
\bal
\int_{E}|\xz g_k(v_k)|{\dd} x<\xe.
\eal
Thus, by Vitali's theorem, we conclude
\ba\label{100}
\int_\xO g_{k}(v_{k})\xz {\dd} x\to \int_\xO g(v)\xz {\dd} x.
\ea
Combining all above, we obtain that $v$ satisfies \eqref{weak}.

Now for any $u\in X_G(\xO),$ we have that $u\in X_{G_{k}}(\xO),$ $G_{k}(u)\leq G(u)$ and
\bal
&\frac{1}{p}\int_{\BBR^N}\int_{\BBR^N}|v_{k}(x)-v_{k}(y)|^pK(x,y){\dd} x{\dd} y+ \int_{\xO}G_{k}(v_{k}){\dd} x-<\xm, v_{k}>\\
&\qquad\leq  \frac{1}{p}\int_{\BBR^N}\int_{\BBR^N}|u(x)-u(y)|^pK(x,y){\dd} x{\dd} y+ \int_{\xO}G_{k}(u){\dd}\xs-<\xm, u>,
\eal
where $G_{k}(r)=\int_0^rg_{k}(s){\dd} s.$ By the above inequality and Fatou's Lemma, we deduce that $v$ is a minimizer of $J$ in $\mathbf{X}_G(\xO).$

Let $g$ be nondecreasing and $u_\xn$ be the minimizer of $J$ associated with $\xn\in (W^{s,p}_0(\xO))^*,$ such that $\xn\leq \xm.$ Then, using $v_k=\min\{(u_\xn-u_\xm)_+,k\}$ as test function, we have that

\bal
&\int_{\BBR^N}\int_{\BBR^N}|u_\xn(x)-u_\xn(y)|^{p-2}(u_\xn(x)-u_\xn(y))(v_k(x)-v_k(y))K(x,y){\dd} x{\dd} y\\
&\qquad-\int_{\BBR^N}\int_{\BBR^N}|u_\xm(x)-u_\xm(y)|^{p-2}(u_\xm(x)-u_\xm(y))(v_k(x)-v_k(y))K(x,y){\dd} x{\dd} y\\
&\qquad\qquad=- \int_\xO (g(u_\xn)-g(u_\xm))v_k {\dd} x+<\xn-\xm,v_k>\leq0.
\eal
Letting $k\to\infty$ in the above inequality and then proceeding as in the proof of \cite[Lemma 6]{KKP}, we  obtain that $u_\xn\leq u_\xm$ a.e. in $\BBR^N.$
\end{proof}

When $\xm\in L^{p'}(\xO),$ we derive the following result which will  be useful in the next subsection.
\begin{lemma}\label{comp}
Let $\xm\in L^{p'}(\xO),$ $g\in C(\BBR^N)$ be a nondecreasing function with $g(0)=0$ and $u\in W^{s,p}_0(\xO)$  satisfy \eqref{weak}. Then there holds,
\ba\label{65}
\int_\xO |g(u)|{\dd} x\leq \int_{\xO}|\xm|{\dd} x.
\ea
In addition, if we assume that $\xm\geq0,$ then $u\geq0$ a.e. in $\BBR^N.$
\end{lemma}
\begin{proof}
Let $k>0.$  Using $\xf_k=\tanh(ku)$ as test function in \eqref{weak}, we obtain
\bal\BAL
\int_{\BBR^N}\int_{\BBR^N}|u(x)-u(y)|^{p-2}(u(x)-u(y))(\xf_k(x)-\xf_k(y))K(x,y){\dd} x{\dd} y+ \int_\xO g(u_\xm)\xf_k {\dd} x=\int_{\xO}\xm\xf_k {\dd} x
\EAL
\eal
If $\infty>u(x)>u(y)>-\infty,$ then there exists $\xi\in (u(y),u(x))$ such that

\bal
\xf_k(x)-\xf_k(y)=(1-\tanh^2(k\xi))(u(x)-u(y))\geq c(\xi,k)(u(x)-u(y)).
\eal
Combining the last two displays, we can easily obtain that
\bal
\int_\xO g(u_\xm)\xf_k {\dd} x\leq\int_{\xO}|\xm| {\dd} x
\eal
By Fatou's lemma and the above inequality, we can easily deduce \eqref{65}.
\end{proof}

\subsection{Subcritical nonlinearities}\label{secsubcr} In this subsection, we always assume that $s\in(0,1),$ $1<p<\frac{N}{s}$ and $g\in C(\BBR)$ is nondecreasing such that $g(0)=0.$

\begin{lemma}\label{prop4}
Let $\xl_i\in \mathfrak{M}_b^+(\xO)$ ($i=1,2$). Then there exist very weak solutions $u,u_i$ ($i=1,2$) to problems

\ba\label{1a}
\left\{
\BAL
Lu +g(u)&=\xl_{1}-\xl_{2}\quad&&\text{in}\;\; \xO\\
u&=0\quad &&\text{in}\;\BBR^N\setminus\xO,
\EAL\right.
\ea
\ba\label{1b}
\left\{
\BAL
L u_1 +g(u_1)&=\xl_{1}\quad&&\text{in}\;\; \xO\\
u&=0\quad &&\text{in}\;\BBR^N\setminus\xO
\EAL\right.
\ea
and

\ba\label{1c}
\left\{
\BAL
L u_{2}-g(-u_2)&=\xl_{2}\quad&&\text{in}\;\; \xO\\
u&=0\quad &&\text{in}\;\BBR^N\setminus\xO,
\EAL\right.
\ea
such that there hold

\ba\label{36a}
u_1,u_2\geq0\quad\text{and}\quad-u_2\leq u\leq u_1\quad\text{a.e. in}\;\;\BBR^N.
\ea

In addition, for any $q\in(0,\frac{N(p-1)}{N-s})$ and $h\in(0,s),$ there exists  a positive constant $c=c(N,p,s,\xL_K, q,h,|\xO|)$ such that 

\ba\label{est2nonlinear}\BAL
\left(\int_\xO |g(u)|{\dd} x\right)^\frac{1}{p-1}&+\left(\int_{\BBR^N}\int_{\BBR^N}\frac{|u(x)-u(y)|^q}{|x-y|^{N+hq}}{\dd} x{\dd} y\right)^\frac{1}{q}\leq c(\xl_1(\xO)+\xl_2(\xO))^\frac{1}{p-1}
\EAL
\ea

and

\ba\label{est3nonlinear}\BAL
\left(\int_\xO |g((-1)^{i+1}u_i)|{\dd} x\right)^\frac{1}{p-1} &+\left(\int_{\BBR^N}\int_{\BBR^N}\frac{|u_i(x)-u_i(y)|^q}{|x-y|^{N+hq}}{\dd} x{\dd} y\right)^\frac{1}{q}\leq c\xl_i(\xO)^\frac{1}{p-1}.
\EAL
\ea

Finally, there exist very weak solutions $v_i$ to \eqref{weaksol} with $\xm=\xl_i$ (i=1,2) such that

\ba\label{36b}
0\leq u_i\leq v_i\leq C_i W_{s,p}^{2\diam(\xO)}[\xl_i]\quad\text{a.e. in}\;\; \xO,
\ea
where $C_i$ is a positive constant depending only on $p,s,\xL_K$ and $N.$
\end{lemma}

\begin{proof}
 Let $\{\xr_n\}_{1}^\infty$ be a sequence of mollifiers  and $\xl_{n,i}=\xr_n*\xl_i.$ Then $\xl_{n,i}\in C_0^\infty(\BBR^N).$ By Proposition \ref{var}, there exist unique solutions $u_n,u_{n,i},v_{n,i}\in W^{s,p}_0(\xO)$ to the following problems

\bal
\left\{
\BAL
L u_n +g(u_n)&=\xl_{n,1}-\xl_{n,2}\quad&&\text{in}\;\; \xO\\
u&=0\quad\quad &&\text{in}\;\; \BBR^N\setminus\xO,
\EAL\right.
\eal

\bal
\left\{
\BAL
L u_{n,1} +g(u_{n,1})&=\xl_{n,1}\quad\text{in}\;\; \xO\\
u&=0,\quad\quad \text{in}\;\; \BBR^N\setminus\xO,
\EAL\right.
\eal

\bal
\left\{
\BAL
L u_{n,2} -g(-u_{n,2})&=\xl_{n,2}\quad\text{in}\;\; \xO\\
u&=0,\quad\quad \text{in}\;\; \BBR^N\setminus\xO
\EAL\right.
\eal
and

\bal
\left\{
\BAL
L v_{n,i}&=\xl_{n,i}\quad&&\text{in}\;\; \xO\\
u&=0,\quad\quad &&\text{in}\;\; \BBR^N\setminus\xO,
\EAL\right.
\eal
such that there holds

\ba\label{36c}
-v_{n,2}\leq-u_{n,2}\leq u_n\leq u_{n,1}\leq v_{n,1}\quad\text{a.e. in} \;\;\BBR^N.
\ea

By Lemma \ref{comp} and Proposition \ref{lpest}, for any $q\in(0,\frac{N(p-1)}{N-s})$ and $h\in(0,s),$ there exists a positive constant $c=c(N,p,s,\xL_K, q,h,|\xO|)$ such that 

\ba\label{12}\BAL
\left(\int_\xO |g(u_n)|{\dd} x\right)^\frac{1}{p-1}+\left(\int_{\BBR^N}\int_{\BBR^N}\frac{|u_n(x)-u_n(y)|^q}{|x-y|^{N+hq}}{\dd} x{\dd} y\right)^\frac{1}{q}\leq c\bigg(\int_\xO\xl_{n,1}+\xl_{n,2}{\dd} x\bigg)^\frac{1}{p-1},
\EAL
\ea

\ba\label{13}\BAL
\left(\int_\xO |g(u_{n,1})|{\dd} x\right)^\frac{1}{p-1}+\left(\int_{\BBR^N}\int_{\BBR^N}\frac{|u_{n,1}(x)-u_{n,1}(y)|^q}{|x-y|^{N+hq}}{\dd} x{\dd} y\right)^\frac{1}{q}\leq c\bigg(\int_\xO\xl_{n,1}{\dd} x\bigg)^\frac{1}{p-1},
\EAL
\ea
\ba\label{14}\BAL
\left(\int_\xO |g(-u_{n,2})|{\dd} x\right)^\frac{1}{p-1}+\left(\int_{\BBR^N}\int_{\BBR^N}\frac{|u_{n,2}(x)-u_{n,2}(y)|^q}{|x-y|^{N+hq}}{\dd} x{\dd} y\right)^\frac{1}{q}\leq c\bigg(\int_\xO\xl_{n,2}{\dd} x\bigg)^\frac{1}{p-1}
\EAL
\ea
and
\ba\label{15}
\left(\int_{\BBR^N}\int_{\BBR^N}\frac{|v_{n,i}(x)-v_{n,i}(y)|^q}{|x-y|^{N+hq}}{\dd} x{\dd} y\right)^\frac{1}{q}
\leq c\bigg(\int_\xO\xl_{n,i}{\dd} x\bigg)^\frac{1}{p-1}.
\ea
Furthermore, in view of the proof of \eqref{estn2}, we have that $T_k(u_n),T_k(u_{n,i}),T_k(v_{n,i})\in W^{s,p}_0(\xO)$ and satisfy \eqref{n2} with $\xm=\xl_1+\xl_2.$

Since the sequences $\{\xl_{n,i}\}_{n}$ are uniformly bounded in $\mathfrak{M}_b(\xO),$ as in the proof of Proposition \ref{sola}, we may show that there exist subsequences, still denoted by the same index, such that $u_n\to u,$ $u_{n,i}\to u_i$ $v_{n,i}\to v_i$ in $W^{h,q}(\BBR^N)$ and a.e. in $\BBR^N.$ In addition, we may prove that $T_k(u),T_k(u_{i}),T_k(v_{i})\in W^{s,p}_0(\xO)$ for any $k>0.$
Finally, by dominated convergence theorem, we deduce that $g(u_{n})\to g(u),$ $g(u_{n,1})\to g(u_{1}),$ $g(-u_{n,2})\to g(-u_{2})$ in $L^1(\xO)$. Hence, combining all above, we can easily show that $u,u_i$ are very weak solutions of problems \eqref{1a}-\eqref{1c} respectively and $v_i$ are very weak solutions of problem \eqref{weaksol} with $\xm=\xl_i$ $(i=1,2).$

By proceeding as in the proof of Proposition \ref{wolffest} and using \eqref{36c}, we derive \eqref{36b}.

Estimates \eqref{est2nonlinear} and \eqref{est3nonlinear} follow by \eqref{12}, \eqref{13} and Fatou's lemma
\end{proof}

\begin{lemma}\label{prop5}
Let  $\xl_i\in \mathfrak{M}_b^+(\xO)$ for $i=1,2.$ We also assume that $g((-1)^{1+i} C W_{s,p}^{2R}[\xl_i])\in L^1(\xO),$ where $C$ is the constant in Proposition \ref{wolffest}. Then the conclusion of Lemma \ref{prop4} holds true.
\end{lemma}
\begin{proof}
Let $T_n(t)=\max(-n,\min(t,n))$ for any $n\in\BBN.$ By Lemma \ref{prop4}, there exist very weak solutions $u_n,u_{n,i}, v_{n,i}\in W^{s,p}_0(\xO)$ of the following problems

\bal
\left\{
\BAL
L u_n +T_nog(u_n)&=\xl_1-\xl_2\quad&&\text{in}\;\; \xO\\
u&=0\quad &&\text{in}\;\; \BBR^N\setminus\xO,
\EAL\right.
\eal

\bal
\left\{
\BAL
L u_{n,1} +T_nog(u_{n,1})&=\xl_1\quad&&\text{in}\;\; \xO\\
u&=0\quad\text{in}\;\; \BBR^N\setminus\xO,
\EAL\right.
\eal

\bal
\left\{
\BAL
L u_{n,2} -T_nog(-u_{n,2})&=\xl_2\quad&&\text{in}\;\; \xO\\
u&=0\quad &&\text{in}\;\; \BBR^N\setminus\xO
\EAL\right.
\eal
and
\bal
\left\{
\BAL
L v_{i}&=\xl_{i}\quad&&\text{in}\;\; \xO\\
u&=0\quad\quad &&\text{in}\;\; \BBR^N\setminus\xO,
\EAL\right.
\eal
such that there holds

\bal
-C W_{1,p}^{2\diam(\xO)}[\xl_2]\leq-v_{2}\leq-u_{n,2}\leq u_n\leq u_{n,1}\leq v_{1}\leq C W_{1,p}^{2\diam(\xO)}[\xl_1]\quad\text{a.e. in}\;\; \BBR^N
\eal
and for any $n\in \BBN.$  The rest of the proof can proceed similarly to the proof of Lemma \ref{prop4} and we omit it.
\end{proof}

\begin{proposition} \label{subcrcon} Assume
\ba \label{subcd0}
\xL_{g}:=\int_1^\infty  s^{-\tilde q-1} (g(s)-g(-s)) {\dd} s<\infty
\ea
for $\tilde q>0.$ Let $v$ be a measurable function defined in $\xO.$ For $s>0$, set
\bal E_s(v):=\{x\in \xO:| v(x)|>s\} \quad \text{and} \quad e(s):=|E_s(v)|.
\eal
Assume that there exists a positive constant $C_0$ such that
\ba \label{e}
e(s) \leq C_0s^{-\tilde q} \quad \forall s\geq1.
\ea
Then for any $s_0\geq1,$ there hold
\bal
\norm{g(|v|)}_{L^1(\Omega)}&\leq \int_{\Omega \setminus E_{s_0}(v)} g(|v|)d x+\tilde q C_0\int_{s_0}^\infty  s^{-\tilde q-1} g(s) {\dd} s, \\
\norm{g(-|v|)}_{L^1(\Omega)}&\leq  -\int_{\Omega \setminus E_{s_0}(v)} g(-|v|){\dd} x-\tilde q C_0\int_{s_0}^\infty  s^{-\tilde q-1}g(-s) {\dd} s.
\eal
\end{proposition}
\begin{proof}
The proof is very similar to the one of \cite[Lemma 5.1]{GkiNg_source} and we omit it.
\end{proof}

\begin{proof}[\textbf{Proof of Theorem \ref{subcritical}}]
Let $\xl_1=\xm^+$ and $\xl_2=\xm^-.$ By Lemma \ref{prop4}, there exist very weak solutions $u_n, v_{i}$ of the following problems

\bal
\left\{
\BAL
L u_n +T_nog(u_n)&=\xl_1-\xl_2\quad&&\text{in}\;\; \xO\\
u&=0\quad &&\text{in}\;\; \BBR^N\setminus\xO
\EAL\right.
\eal
and
\bal
\left\{
\BAL
L v_{i}&=\xl_{i}\quad&&\text{in}\;\; \xO\\
u&=0\quad\quad &&\text{in}\;\; \BBR^N\setminus\xO,
\EAL\right.
\eal
such that there holds

\bal
-v_{2}\leq u_n\leq v_{1}\quad\text{a.e. in}\;\; \BBR^N\;\;\text{and}\;\;\forall n\in \BBN.
\eal
 Also, taking into consideration that $g$ in nondecreasing with $g(0)=0$, we may show that $T_k(u_n),T_k(v_i)$ satisfy \eqref{n2} with $\xm=\xl_1+\xl_2.$ In addition, by \eqref{estn1}, there holds

\bal
\norm{v_1^{p-1}}^{*}_{L^{\frac{N}{N-sp}}_w(\BBR^N)}+ \norm{v_2^{p-1}}^{*}_{L^{\frac{N}{N-sp}}_w(\BBR^N)}\leq C_1(N,p,s,\xL_K)(\xl_1(\xO)+\xl_2(\xO)).
\eal

By \eqref{ue} and Proposition \ref{subcrcon}, we have that $|T_nog(u_n)|\leq g(v_1)-g(-v_2)$ and

\bal
&\norm{T_nog(u_n)}_{L^1(\Omega)}\leq \norm{g(v_1)}_{L^1(\Omega)}+\norm{g(-v_2)}_{L^1(\Omega)}\\
&\;\;\leq (g(s_0)-g(s_0))|\xO|\\
&\qquad+\tilde q C_1(N,p,s,\xL_K,\xL_g)(\xl_1(\xO)+\xl_2(\xO))^\frac{N(p-1)}{N-sp}\int_{s_0}^\infty  s^{-\tilde q-1} (g(s)-g(-s) {\dd} s\quad\forall n\in \BBN,
\eal
where $\tilde q=\frac{N(p-1)}{N-sp}.$ The desired result follows by proceeding as in the proof of Lemma \ref{prop4}.
\end{proof}

\subsection{Power nonlinearities: Proof of Theorem \ref{powerabsorptionsupercritical} }\label{secpower}

In order to prove Theorem \ref{powerabsorptionsupercritical}, we need to introduce some notations concerning the Bessel capacities, we refer the reader to \cite{Adbook} for more detail. For $\ga\in\BBR$ we define the Bessel kernel of order $\ga$ by $G_{\ga}(\xi)=\CF^{-1}(1+|.|^2)^{-\frac{\ga}{2}}(\xi)$, where $\CF$ is the Fourier transform of moderate distributions in $\BBR^N$. For any $\xb>1, $ the Bessel space $L_{\ga,\xb}(\BBR^N)$ is given by
\bal L_{\ga,\xb}(\BBR^N):=\{f=G_{\alpha} \ast g:g\in L^{\xb}(\BBR^N)\},
\eal
with norm
\bal \|f\|_{L_{\xa,\xb}(\BBR^N)}:=\|g\|_{L^\xb(\BBR^N)}=\|G_{-\ga}\ast f\|_{L^\xb(\BBR^N)}.
\eal
The Bessel capacity is defined as follows.
 \begin{definition}
 Let $\xa>0,$ $1< \xb<\infty$ and $E\subset\BBR^N.$ Set
 \bal
 \mathcal{S}_E:=\{g\in L^\xb(\BBR^N):\;g\geq0,\;G_{\alpha} \ast g(x)\geq 1\;\;\text{for any}\;x\in E\}.
 \eal
Then
\ba\label{besselcapacity}
\mathrm{Cap}_{{\ga,\xb}}(E):=\inf\{\|g\|^\xb_{L^\xb(\BBR^N)}; g\in \mathcal{S}_E \}.
\ea
If $\mathcal{S}_E=\emptyset,$ we set $\mathrm{Cap}_{{\ga,\xb}}(E)=\infty.$
\end{definition}
\noindent In the sequel, we denote by $L_{-\xa,\xb'}(\BBR^N)$ the dual of $L_{\xa,\xb}(\BBR^N)$ and  we set
\bal
\BBG_{\ga}[\xm](x)=\int_{\BBR^N}G_{\ga}(x,y){\dd} \xm(y)\quad\forall \xm\in \mathfrak{M}(\mathbb{R}^N).
\eal

\begin{proof}[\textbf{Proof of Theorem \ref{powerabsorptionsupercritical}}]
Since $\xm$ is absolutely continuous with respect to the capacity \small $\mathrm{Cap}_{{sp,\frac{\xk}{\xk-p+1}}},$ \normalsize the measures $\xm^+,\xm^-$ have the same property. Thus, by \cite[Theorem 2.5]{BNV} (see also \cite{BP}), there are nondecreasing sequences $\{\xm_{n}^\pm\}_n\subset L^{-sp,\frac{\xk}{p-1}}(\BBR^N)\cap\mathfrak{M}^+_b(\BBR^N)$ with compact support in $\xO,$ such that
they converge to $\xm^\pm$ in the narrow topology. Furthermore, by \cite[Theorem 2.3]{BNV} (see also \cite[Corollary 3.6.3]{Adbook}), 

\bal
\norm{ W^{2\diam(\xO)}_{\xa, p}[\gm_n^\pm]}_{L^{\xk}(\BBR^N)}^{\xk}\approx \norm{\BBG_{sp}[\xm_n^\pm]}_{L^{\frac{\xk}{p-1}}(\BBR^N)}^{\frac{\xk}{p-1}}<\infty.
\eal

By Lemma \ref{prop5}, there exist solutions $u_n,u_{n,i},v_i$ to the problems

\ba\label{38}
\left\{
\BAL
Lu_n +|u_n|^{\xk-1}u_n&=\xl_{n,1}-\xl_{n,2}\quad&&\text{in}\;\; \xO\\
u&=0\quad &&\text{in}\;\; \BBR^N\setminus\xO,
\EAL\right.
\ea

\ba\label{39}
\left\{
\BAL
L u_{n,1} +|u_{n,1}|^{\xk-1}u_{n,1}&=\xl_{n,1}\quad&&\text{in}\;\; \xO\\
u&=0,\quad &&\text{in}\;\; \BBR^N\setminus\xO,
\EAL\right.
\ea

\ba\label{39b}
\left\{
\BAL
L u_{n,2} +|u_{n,2}|^{\xk-1}u_{n,2}&=\xl_{n,2}\quad&&\text{in}\;\; \xO\\
u&=0\quad\quad &&\text{in}\;\; \BBR^N\setminus\xO
\EAL\right.
\ea
and
\bal
\left\{
\BAL
L v_{n,i}&=\xl_{n,i}\quad&&\text{in}\;\; \xO\\
u&=0\quad\quad &&\text{in}\;\; \BBR^N\setminus\xO,
\EAL\right.
\eal
such that there holds

\ba\label{42}\BAL
 -v_{n,2}\leq-u_{n,2}&\leq u_n\leq u_{n,1}\leq v_{n,1}\quad\text{a.e. in}\;\;\BBR^N.
\EAL
\ea
Furthermore,  in view of the proof of Lemmas \ref{prop4} and \ref{prop5}, the sequences $\{u_{n,i}\},\{v_{n,i}\}$ satisfy \eqref{12}-\eqref{15} with $g(t)=|t|^\xk\sign(t),$ $\xl_{n,1}=\xm_n^+$ and $\xl_{n,2}=\xm_n^-,$ as well as they can be constructed such that

\ba\label{42b}
u_{n,i}\leq u_{n+1,i}\quad\text{and}\quad v_{n,i}\leq v_{n+1,i}\quad \text{a.e. in}\;\BBR^N, \forall n\in\BBN\;\text{and}\;i=1,2.
\ea

By \eqref{12}-\eqref{13} with $g(t)=|t|^\xk\sign(t),$ $\xl_{n,1}=\xm_n^+$ and $\xl_{n,2}=\xm_n^-,$ we have
\bal
\int_\xO |u_{n,1}|^\xk{\dd}\leq \xm^+(\xO)\quad\text{and}\quad \int_\xO |u_{n,2}|^\xk{\dd}\leq \xm^-(\xO)\quad\forall n\in\BBN.
\eal
By \eqref{12}-\eqref{15} with $g(t)=|t|^\xk\sign(t),$ $\xl_{n,1}=\xm_n^+$ and $\xl_{n,2}=\xm_n^-,$ there are  subsequences, still denoted by the same index, such that $u_n\to u,$ $u_{n,i}\to u_i$ $v_{n,i}\to v$ in $W^{h,q}(\BBR^N)$ and a.e. in $\BBR^N.$ In addition, $T_k(u),T_k(u_i), T_k(v_i)\in W^{s,p}_0(\BBR^N)$ and

\bal
\int_\xO |u_1|^k{\dd} x\leq \xm^+(\xO)\quad\text{and}\quad \int_\xO |u_2|^k{\dd} x\leq \xm^-(\xO).
\eal
Therefore, by dominated convergence theorem, we obtain that $|u_{n}|^\xk\to |u|^\xk,$ $|u_{n,1}|^\xk\to |u_{1}|^\xk,$ $|u_{n,2}|^\xk\to |u_{2}|^\xk$ in $L^1(\xO)$. This, implies that $u,u_i$ are very weak solutions of problems \eqref{1a}-\eqref{1c} respectively and $v_i$ are very weak solution  of problem \eqref{weaksol} with $\xm=\xl_i,$ where  $\xl_1=\xm^+$ and $\xl_2=\xm^-.$

Estimate \eqref{41} follows by \eqref{42} and \eqref{36b}. Estimate \eqref{est2nonlinearpower} follows by \eqref{12} with $g(t)=|t|^\xk\sign(t),$ $\xl_{n,1}=\xm_n^+,$ $\xl_{n,2}=\xm_n^-$ and Fatou's lemma

\end{proof}

\section{Nonlocal equations with source nonlinearities}\label{secsource}

\subsection{Subcritical nonlinearities}\label{secsourcesubcr}
In this subsection, we investigate the existence of solutions to the following problem

\ba\label{trans} \left\{
\BAL
L v&=g(v)+\xr\tau\quad&&\text{in}\;\; \xO\\
v&=0\quad &&\text{in}\;\;\BBR^N\setminus\xO,
\EAL
\right.
\ea
where $\xr>0,$ $g\in C(\BBR)$ is a nondecreasing function and

\ba \label{gcomparepower}
 |g(t)|\leq a|t|^d \quad \text{for some } a>0,\; d>p-1 \text{ and for any } |t|\leq 1.
 \ea

Let us state the first existence result.
\begin{lemma} Let $1<p<\frac{s}{N}$ and $\tau \in C_0^\infty(\BBR^N)$ be such that  $\norm{\gt}_{L^1(\BBR^N)}\leq 1.$ Assume that $g \in L^\infty(\xO) \cap C(\R)$ satisfies \eqref{subcd0} for
\bal
\tilde q=\frac{N(p-1)}{N-sp}.
\eal
In addition, we assume that $g$ is nondecreasing and satisfies \eqref{gcomparepower}.

Then there exists a positive constant $\xr_0$ depending on $N,\xO,\Gl_{g},\xL_K,a,d,p,s$ such that for every $\xr \in (0,\xr_0),$  problem \eqref{trans}
admits a weak solution $v\in W^{s,p}_0(\xO)$ satisfying
	\ba\label{estM}
\||v|^{p-1}\|_{L_w^{\frac{N}{N-sp}}(\Gw)} \leq t_0,
	\ea
	where $t_0>0$ depends on $N,\xO,\Gl_{g},\xL_K,a,d,p,s$.
\end{lemma}
\begin{proof}
We shall use Schauder fixed point theorem to show the existence of a positive weak solution of \eqref{trans}.

Let $1<\kappa<\min\{\frac{N}{N-sp},\frac{d}{p-1}\}$ and $v\in L^1(\xO).$ Since $g\in L^\infty(\xO),$ we can easily show that the following problem
\ba\label{trans2} \left\{
\BAL
Lu&=g(|v|^{\frac{1}{p-1}}\sign(v))+\xr\tau\quad&&\text{in}\;\; \xO\\
u&=0\quad&&\text{in}\;\; \BBR^N\setminus\xO
\EAL
\right.
\ea
admits a unique weak solution $\BBT(v)\in W_0^{s,p}(\xO).$

We define the operator $\BBS$ by
	\be\label{Sn} \BBS(v):=|\BBT(v)|^{p-1}\sign(\BBT(v))\quad \forall  v \in L^{1}(\Gw).
	\ee
By \eqref{est7}, we obtain

\ba\label{bound}\BAL
\|\BBS(v)\|_{L_w^{\frac{N}{N-sp}}(\Gw)} &\leq C(s,p,N,\xL_K) \left(\xr\int_\xO|\tau|{\dd} x+\int_{\xO}|g(|v|^{\frac{1}{p-1}}\sign(v))|{\dd} x\right)\\
&\leq C(s,p,N,\xL_K)\left(\xr+\int_{\xO}g(|v|^{\frac{1}{p-1}})-g(-|v|^{\frac{1}{p-1}}){\dd} x\right).
\EAL
\ea

Let $v \in L_w^{\frac{N}{N-sp}}(\Gw)$. For any $\xl>0$, we set $E_\xl:=\{x\in\xO: |v(x)|^{\frac{1}{p-1}}>\xl\}$ and $e(\xl)=\int_{E_\xl}{\dd} x$. By \eqref{semi} and \eqref{equinorm}, we can easily show that
\bal
e(\xl)\leq C(N,s,p)\norm{v}_{L_w^{\frac{N}{N-sp}}(\Gw)}^{\frac{N}{N-sp}}\xl^{-\frac{N(p-1)}{N-sp}}.
\eal
By the above inequality and Lemma \ref{subcrcon} with $\xl_0=1$ and $\tilde q=\frac{N(p-1)}{N-sp},$ we deduce

\bal
\int_{\xO}g(|v|^{\frac{1}{p-1}})-g(-|v|^{\frac{1}{p-1}}){\dd} x&\leq 2a \int_{\xO} |v|^{\xk}{\dd} x +C(p,s,N)\norm{v}_{L_w^{\frac{N}{N-sp}}(\Gw)}^\frac{N}{N-sp} \xL_g.
\eal
Let $\xl=\norm{v}_{L_w^{\frac{N}{N-sp}}(\Gw)}.$ By \eqref{equinorm}, we have that

\bal
\int_{\xO} |v|^{\xk}{\dd}&=\int_0^\infty |\{x\in\xO:\;|v|\geq t\}|{\dd} t^\xk\\
&=\int_0^\xl  |\{x\in\xO:\;|v|\geq t\}|{\dd} t^\xk+\int_\xl^\infty  |\{x\in\xO:\;|v|\geq t\}|{\dd} t^\xk\\
&\leq |\xO|\xl^\xk+\xk\xl^{\frac{N}{N-sp}}\int_{\xl}^\infty t^{\xk-\frac{N}{N-sp}-1} {\dd} t \leq C(\xO,\xk,s,p,N)\xl^k.
\eal
Combining all above, we may prove that

\bal
\|\BBS(v)\|_{L_w^{\frac{N}{N-sp}}(\Gw)} \leq C(p,N,\xk,|\xO|,\xL_{g},\xL_K,a) \bigg(\xr+\norm{v}_{L_w^{\frac{N}{N-sp}}(\Gw)}^\frac{N}{N-sp}+\norm{v}_{L_w^{\frac{N}{N-sp}}(\Gw)}^\xk\bigg).
\eal
Therefore, if $\norm{v}_{L_w^{\frac{N}{N-sp}}(\Gw)}\leq t$ then
\ba \label{Qt0} \|\BBS(v)\|_{L_w^{\frac{N}{N-sp}}(\Gw)} \leq C\left(t^{\frac{N}{N-sp}}+t^{\xk} +\xr\right).
\ea
Since $1<\xk<\frac{N}{N-sp}$, there exist $t_0>0$ and $\rho_0>0$  depending on $|\Gw|,\Gl_g,p,\xk,N,a$ such that for any
$t\in(0,t_0]$ and $\rho \in (0,\rho_0),$ the following inequality holds
\bal
C\left(t^{\frac{N}{N-sp}}+t^{\xk}+\xr\right) \leq t_0,
\eal
where $C$ is the constant in \eqref{Qt0}. Hence,
	\bel{ul11}
\|v\|_{L_w^{\frac{N}{N-sp}}(\Gw)} \leq t_0  \Longrightarrow \|\BBS(v)\|_{L_w^{\frac{N}{N-sp}}(\Gw)}\leq t_0.
	\ee

	\noindent \textbf{Step 3:} We apply Schauder fixed point theorem to our setting.
	
\medskip
	
	\textit{We claim that $\BBS$ is continuous}. First we assume that $v_n\rightarrow v$ in $L^1(\Gw)$ and $\BBT(v_n)\to \BBT(v)$ in $W^{1,p}_0(\xO),$ then by fractional Sobolev inequality, we have

\ba\label{exisub1}\BAL
\int_\xO |\BBT(v_n)-\BBT(v)| {\dd} x&\leq |\xO|^{\frac{pN-N+sp}{N p}}\norm{\BBT(v_n)-\BBT(v)}_{L^{\frac{Np}{N-sp}}(\xO)}\\
&\leq C|\xO|^{\frac{pN-N+sp}{N p}}\norm{\BBT(v_n)-\BBT(v)}_{W^{1,p}_0(\xO)}\to 0.
\EAL
\ea

Let $k>0$ and $\xe>0,$ then

\ba\label{exisub2}\BAL
\int_\xO |\BBS(v_n)-\BBS(v)| {\dd} x&=\int_{\{x\in\xO:\;|\BBS(v_n)(x)|\leq k\}\cap\{x\in\xO:\;|\BBS(v)(x)|\leq k\}}\left|\BBS(v_n)-\BBS(v)\right|{\dd} x\\
&+\int_{\xO\setminus(\{x\in\xO:\;|\BBS(v_n)(x)|\leq k\}\cap\{x\in\xO:\;|\BBS(v)(x)|\leq k\})}\left|\BBS(v_n)(x)-\BBS(v)(x)\right|{\dd} x.
\EAL
\ea

By \eqref{bound} and the fact that $g\in L^\infty(\BBR),$ we have that $\BBS(v_n)\in L^{\xb}(\xO) $ and $\{\BBS(v_n)\}$ is uniformly bounded in $L^{ \xb}(\xO)$ for any $\xb\in (1,\frac{N}{N-sp}).$ Hence, there exists $k_0\in\BBN,$ such that
\ba\label{exisub3}\BAL
\int_{\xO\setminus(\{x\in\xO:\;|\BBS(v_n)(x)|\leq k\}\cap\{x\in\xO:\;|\BBS(v)(x)|\leq k\})}\left|\BBS(v_n)-\BBS(v)\right|{\dd} x\leq \frac{\xe}{3}\quad\forall k\geq k_0\quad\text{and}\quad n\in\BBN.
\EAL
\ea

Now, we set
\bal
A_{k_0,n}=\{x\in\xO:\;|\BBT(v_n)(x)|\leq k^{\frac{1}{p-1}}_0\}\cap\{x\in\xO:\;|\BBT(v)(x)|\leq k^{\frac{1}{p-1}}_0\}
\eal
and $B_{\xd,n}=\{x\in\xO:\;|\BBT(v)(x)-\BBT(v_n)(x)|\leq \xd\}.$ Then, we have that

\ba\label{exisub4}\BAL
&\int_{\xO\cap\{x\in\xO:\;|\BBS(v_n)|\leq k_0\}\cap\{x\in\xO:\;|\BBS(v)|\leq k_0\}}\left|\BBS(v_n)-\BBS(v)\right|{\dd} x\\
&=\int_{A_{k_0,n}\cap B_{\xd,n}}\left||\BBT(v_n)|^{p-1}\sign(\BBT(v_n))-|\BBT(v)|^{p-1}\sign(\BBT(v))\right|{\dd} x\\
&+\int_{A_{k_0,n}\setminus B_{\xd,n}}\left||\BBT(v_n)|^{p-1}\sign(\BBT(v_n))-|\BBT(v)|^{p-1}\sign(\BBT(v))\right|{\dd} x.
\EAL
\ea
Since $h(t)=t^{p-1}\sign(t)$ is uniformly continuous in $[-k_0,k_0],$ there exists $\xd_0>0$ independent of $n$ such that
\ba\label{exisub5}\BAL
\int_{A_{k_0,n}\cap B_{\xd_0,n}}\left||\BBT(v_n)|^{p-1}\sign(\BBT(v_n))-|\BBT(v)|^{p-1}\sign(\BBT(v))\right|{\dd}x\leq \frac{\xe}{3}.
\EAL
\ea
Moreover, by \eqref{exisub1}, there exists $n_0=n_0(\xd_0,k_0,p)\in \BBN$ such that
\ba\label{exisub6}
\int_{A_{k_0,n_0}\setminus B_{\xd_0,n_0}}\left||\BBT(v_{n_0})|^{p-1}\sign(\BBT(v_{n_0}))-|\BBT(v)|^{p-1}\sign(\BBT(v))\right|{\dd}x\leq \frac{\xe}{3}.
\ea
Hence, combining \eqref{exisub1}-\eqref{exisub6}, we obtain that $\BBS(v_n)\to\BBS(v)$ in $L^1(\xO).$

Therefore, it is enough to show that $\BBT(v_n)\to \BBT(v)$ in $W^{s,p}_0(\xO).$ In order to prove this, we will consider two cases.

\medskip

\emph{Case 1. $1<p<2.$} Let $M:=\sup_{t\in\BBR}|g(t)|.$ We will show that $\BBT(v_n)\to \BBT(v)$ in $W^{s,p}_0(\xO).$ Since $\BBT(v_n),\BBT(v)\in W_0^{s,p}(\xO)$ are weak solutions of \eqref{trans2} with $v_n$ and $v$ respectively, we have

\ba\label{55}
\BAL
\int_{\BBR^N}\int_{\BBR^N}&|\BBT(v_n)(x)-\BBT(v_n)(y)|^{p}K(x,y){\dd} x{\dd} y=\int_\xO\BBT(v_n)(g(|v_n|^{\frac{1}{p-1}}\sign(v_n)){\dd} x+\int_\xO\BBT(v_n)\tau {\dd} x\\
&\leq M|\xO|^{\frac{p-1}{p}}\left(\int_\xO|\BBT(v_n)|^{p} {\dd} x\right)^{\frac{1}{p}}
+\left(\int_\xO|\BBT(v_n)|^{p} {\dd} x\right)^{\frac{1}{p}}\left(\int_\xO|\tau|^{\frac{p}{p-1}} {\dd} x\right)^{\frac{p-1}{p}}\\
&\leq C_1(M,\xO,p,N,\tau,s) \left(\int_{\BBR^N}\int_{\BBR^N}\frac{|\BBT(v_n)(x)-\BBT(v_n)(y)|^{p}}{|x-y|^{N+sp}}{\dd} x{\dd} y\right)^{\frac{1}{p}}.
\EAL
\ea
Therefore,
\ba\label{54}
\int_{\BBR^N}\int_{\BBR^N}&\frac{|\BBT(v_n)(x)-\BBT(v_n)(y)|^{p}}{|x-y|^{N+sp}}{\dd} x{\dd} y\leq C_1^{\frac{p-1}{p}}(M,\xO,p,N,\tau,s,\xL_K).
\ea

Using $\xf=\BBT(v_n)-\BBT(v)$ as test function, we have
\ba\label{exisub11}\BAL
I&:=\int_{\BBR^N}\int_{\BBR^N}| \BBT(v_n)(x)-\BBT(v_n)(y)|^{p-2}(\BBT(v_n)(x)-\BBT(v_n)(y))\left(\xf(x)-\xf(y)\right)K(x,y){\dd} x{\dd} y\\
&-\int_{\BBR^N}\int_{\BBR^N}| \BBT(v)(x)-\BBT(v)(y)|^{p-2} (\BBT(v)(x)-\BBT(v)(y))\left(\xf(x)-\xf(y)\right)K(x,y){\dd} x{\dd} y\\
&=
\int_\xO\xf(g(|v_n|^{\frac{1}{p-1}}\sign(v_n))-g(|v|^{\frac{1}{p-1}}\sign(v))){\dd} x=:II.
\EAL
\ea

We first treat $I.$ On one hand, since $(|a|^{p-2}a-|b|^{p-2}b)(a-b)\geq C(p)\frac{|a-b|^2}{(|a|+|b|)^{2-p}}$ for any $(a,b)\in\BBR^{2N}\setminus\{(0,0)\}$ and $p\in(1,2),$ we have

\ba\label{exisub7}
I\geq C(p)\int_{\BBR^N}\int_{\BBR^N}\left|\xf(x)-\xf(y)\right|^2\left(|\BBT(v_n)(x)-\BBT(v_n)(y)|+|\BBT(v)(x)-\BBT(v)(y)|\right)^{p-2}K(x,y){\dd} x{\dd} y.
\ea
On the other hand, by H\"{o}lder inequality, we obtain

\ba\label{exisub8}\BAL
&\int_{\BBR^N}\int_{\BBR^N}\frac{\left|\xf(x)-\xf(y)\right|^p}{|x-y|^{N+sp}} {\dd} x {\dd} y\leq \xL_K\int_{\BBR^N}\int_{\BBR^N}\left|\xf(x)-\xf(y)\right|^pK(x,y) {\dd} x {\dd} y\\
&\leq C(p,\xL_K)\left(\int_{\BBR^N}\int_{\BBR^N}\left(|\BBT(v_n)(x)-\BBT(v_n)(y)|+|\BBT(v)(x)-\BBT(v)(y)|\right)^{p} K(x,y){\dd} x {\dd} y\right)^{\frac{2-p}{2}} I^\frac{p}{2}\\
&\leq C(p,C_1,\xO,\xL_K)I^\frac{p}{2},
\EAL
\ea
where $C_1$ is the constant in \eqref{54}. Hence, by \eqref{exisub7} and \eqref{exisub8}, we obtain

\ba\label{exisub9}
C\left(\int_{\BBR^N}\int_{\BBR^N}\frac{\left|\xf(x)-\xf(y)\right|^p}{|x-y|^{N+sp}} {\dd} x {\dd} y\right)^{\frac{2}{p}}\leq  I.
\ea

Next we treat $II.$ Let $r=\frac{N p}{N-sp},$ proceeding as in the proof of \eqref{55}, we have
\ba\label{exisub10}\BAL
&II\leq \left(\int_\xO|\xf|^{r}{\dd}\right)^{\frac{1}{r}}
\left(\int_\xO|g(|v_n|^{\frac{1}{p-1}}\sign(v_n))-g(|v|^{\frac{1}{p-1}}\sign(v))|^{r'}{\dd}\right)^{\frac{1}{r'}}\\
&\leq C(N,p,s) \left(\int_{\BBR^N}\int_{\BBR^N}\frac{\left|\xf(x)-\xf(y)\right|^p}{|x-y|^{N+sp}} {\dd} x {\dd} y\right)^{\frac{1}{p}} \\
&\qquad\qquad\times\left(\int_\xO|g(|v_n|^{\frac{1}{p-1}}\sign(v_n))-g(|v|^{\frac{1}{p-1}}\sign(v))|^{r'}{\dd}\right)^{\frac{1}{r'}},
\EAL
\ea
where in the last inequality we used the fractional Sobolev inequality.

 Combining \eqref{exisub11}, \eqref{exisub9} and \eqref{exisub10}, we obtain

\ba\label{exisub12}\BAL
&\left(\int_{\BBR^N}\int_{\BBR^N}\frac{\left|\xf(x)-\xf(y)\right|^p}{|x-y|^{N+sp}} {\dd} x {\dd} y\right)^{\frac{1}{p}}\\
&\qquad\qquad\qquad\leq C(p,C_1,\xO,s,\xL_K)\left(\int_\xO|g(|v_n|^{\frac{1}{p-1}}\sign(v_n))-g(|v|^{\frac{1}{p-1}}\sign(v))|^{r'}{\dd} \right)^{\frac{1}{r'}}.
\EAL
\ea
Since $g\circ(|\cdot|^{p-1}\sign(\cdot))$ is uniformly continuous in $\BBR,$ bounded and $v_n\to v$ in $L^1(\xO),$ we obtain
\bal
\lim_{n\to\infty}\int_{\xO}|g(|v|^{\frac{1}{p-1}}\sign(v))-g(|v_n|^{\frac{1}{p-1}}\sign(v_n))|^{r'}{\dd} x=0,
\eal
which, together with \eqref{exisub12}, implies the desired result.

\medskip

\emph{Case 2. $p\geq 2.$} We note here that $(|a|^{p-2}a-|b|^{p-2}b)(a-b)\geq C(p)|a-b|^p$ for any $(a,b)\in\BBR^{2N}$ and $p\geq2.$ Thus,

\bal
I\geq C(N,p,\xL_K)\left(\int_{\BBR^N}\int_{\BBR^N}\frac{\left|\xf(x)-\xf(y)\right|^p}{|x-y|^{N+sp}} {\dd} x {\dd} y\right).
\eal
By using a similar argument to the one in Case 1, we may show that $\BBT(v_n)\to \BBT(v)$ in $W^{s,p}_0(\xO).$

\medskip
	
\textit{Next we claim that $\BBS$ is compact}. Indeed, let $\{v_n\}$ be a sequence in $ L^1(\Gw)$ then by \eqref{54},
we obtain that $\BBT(v_n)$ is uniformly bounded in $W^{s,p}_0(\xO).$ Hence there exists  a subsequence still denoted by $\{\BBT(v_n)\}$ such that $\BBT(v_n)\rightharpoonup \psi$ in $W^{s,p}_0(\xO)$ and  $\BBT(v_n)\to \psi$ a.e. in $\BBR^N.$ Furthermore, in view of \eqref{bound}, we can easily show that $\BBS(v_n)=|\BBT(v_n)|^{p-1}\sign(\BBT(v_n))\to |\psi|^{p-1}\sign(\psi)$ in $L^1(\xO).$
	
Now set
	\be
	\CO:=\{ v \in L^1(\Gw):\; \norm{v}_{L_w^{\frac{N}{N-sp}}(\Gw)} \leq t_0  \}.\label{O}
	\ee
	Then $\CO$ is a closed, convex subset of $L^1(\Gw)$ and by \eqref{ul11}, $\BBS(\CO) \sbs \CO$.
	Thus we can apply Schauder fixed point theorem to obtain the existence of a function $v \in \CO$ such that $\BBS(v)=v$. This means that $u=v^{\frac{1}{p-1}}\sign(v)$ is a solution of \eqref{trans} satisfying \eqref{estM}.

\end{proof}

\begin{proof}[\textbf{Proof of Theorem \ref{sourcegeneral}}]
Let $\{\xr_n\}_{n=1}^\infty$ be a sequence of mollifiers. Set $\tau_n=\xr_n*\tau$ and $g_n=\max(-n,\min(g,n)).$ Then $g_n$ satisfies \eqref{cond} with the same constant $\xL_g.$ Thus, there exists a weak solution $u_n\in W^{s,p}_0(\xO)$ of
\bal\left\{
\BAL
Lv&=g_n(v)+\xr\tau_n\quad&&\text{in}\;\; \xO\\
v&=0\quad &&\text{in}\;\; \BBR^N\setminus\xO.
\EAL
\right.
\eal
In addition, it satisfies
	\ba\label{estM1}
\||u_n|^{p-1}\|_{L_w^{\frac{N}{N-sp}}(\Gw)} \leq t_0,
	\ea
where $t_0>0$ depends on $N,\Gw,\Gl_g,\xL_K,a,s,p,d.$

By \eqref{estM1}, we have that $|\{x\in\xO:\;|u_n|>s\}|\leq t_0^{\frac{N}{N-sp}}s^{-\frac{N(p-1)}{N-sp}}.$ Hence by Proposition \ref{subcrcon}
\bal
\int_\xO |g_n(u_n)|{\dd} x\leq C\quad\forall n\in \BBN,
\eal
where $C$ depends only on $N,\Gw,\Gl_g,\xL_K,a,s,p,d$ and $t_0.$ This, together with Proposition \ref{lpest}, implies that  for any $q\in(p-1,\frac{N(p-1)}{N-s})$ and $h\in(0,s),$ there exists a positive constant  $c=c(N,s,p,\xL_K,s,h,q,|\xO|)$ such that

\ba\label{est2nonlinearsourceb}\BAL
\left(\int_{\BBR^N}\int_{\BBR^N}\frac{|u_n(x)-u_n(y)|^q}{|x-y|^{N+hq}}{\dd} x{\dd} y\right)^\frac{1}{q}\leq c(C+\xr\int_\xO|\tau_n|{\dd} x)^\frac{1}{p-1}.
\EAL
\ea
Therefore, in view of the proof of Proposition \ref{sola}, we may show that there exists a subsequence, still denoted by the same notation, such that  $u_n\rightarrow u$ in $W^{h,q}(\BBR^N)$ and a.e. in $\BBR^N.$

Now, we will show that $g_n(u_n)\to g(u)$ in $ L^1(\xO).$ We will prove it by using Vitali's convergence theorem. Let $E\subset \xO$ be a Borel set. Then, by Lemma \ref{subcrcon} and \eqref{estM1}, we have
\bal
\int_E |g_n(u_n)|{\dd} x&\leq \int_\xO |g(u_n)|{\dd} x\\
&\leq (g(s_0)-g(-s_0))|E| + C(t_0,p,\xL_g,N)\int_{s_0}^\infty (g(s)-g(-s))s^{-1-\frac{(p-1)N}{N-sp}}{\dd} s,\;\;\forall s_0\geq1.
\eal
Let $\xe>0,$ then there exists $s_0$ such that

\bal
 C(t_0,p,\xL_g,N)\int_{s_0}^\infty (g(s)-g(-s))s^{-1-\frac{(p-1)N}{N-sp}}{\dd} s\leq \frac{\xe}{2}.
\eal
Set $\xd=\frac{\xe}{2(1+g(s_0)-g(-s_0))}>0.$ Then for any Borel set $E$ with $|E|\leq \xd,$ we have
\bal
g(s_0)|E|\leq \frac{\xe}{2}.
\eal
Hence, by the last three inequalities, we may invoke Vitali's convergence theorem in order to prove that $g_n(u_n)\to g(u)$ in  $L^1(\xO).$

In view of the proof of Proposition \ref{sola}, we may deduce that $u$ is a very weak solution of \eqref{trans0}. Furthermore, by Fatou's lemma, we can easily show that $u$ satisfies \eqref{estM2} and \eqref{est2nonlinearsource}.
\end{proof}

\subsection{Power nonlinearities: Proof of Proposition \ref{exissource} and Theorem \ref{supercrpower1}}\label{secsourcepower}

\begin{proof}[\textbf{Proof of Proposition \ref{exissource}}]
Let $w=ACW_{s,p}^{2\diam(\xO)}[\xr\tau],$ where $C$ is the constant in \eqref{est3} and $A>1$ is a constant that will be determined later. Set ${\dd}\xn=w^\xk {\dd} x+\xr \dd\tau,$ then by \eqref{subcriticalcondition}, we obtain
\bal
CW_{s,p}^{2(\diam(\xO))}[\xn]&\leq 2^{\frac{1}{p-1}}C(W_{s,p}^{2\diam(\xO)}[w^\xk]+W_{s,p}^{2\diam(\xO)}[\xr \dd\tau])\\
&\leq 2^{\frac{1}{p-1}}C((AC)^\frac{\xk}{p-1}\xr^{\frac{\xk}{(p-1)^2}}MW_{s,p}^{2\diam(\xO)}[\tau]+W_{s,p}^{2\diam(\xO)}[\xr \dd\tau])\\
&\leq 2^{\frac{1}{p-1}}C((AC)^\frac{\xk}{p-1}M\xr^{\frac{\xk-p+1}{(p-1)^2}}+1)W_{s,p}^{2\diam(\xO)}[\xr \dd\tau].
\eal
If we choose $A=2^{\frac{1}{p-1}+1}$ and $\xr$ small enough such that $(AC)^\frac{\xk}{p-1}M\xr^{\frac{\xk-p+1}{(p-1)^2}}+1<2,$ we deduce that
\ba\label{16}
CW_{s,p}^{2\diam(\xO)}[\xn]\leq w.
\ea

Now, let $x_0\in\xO$ be such that $W_{s,p}^{2\diam(\xO)}[\tau](x_0)<\infty.$ If $0\leq v\leq c_0W_{s,p}^{2\diam(\xO)}[\tau]$ a.e. in $\BBR^N,$ for some constant $c_0>0,$ then we have
\bal
\left(\int_{\xO}|v|^\xk {\dd} x\right)^{\frac{1}{p-1}} \leq\left(\int_{B_{\diam(\xO)}(x_0)}|v|^\xk {\dd} x\right)^{\frac{1}{p-1}}\leq C(\xO,N,s,p,M,K,c_0)W_{s,p}^{2\diam(\xO)}[\tau](x_0)<\infty.
\eal
Thus $v\in L^\xk(\xO).$

Let $u_0\geq0$ be a very weak solution of 
\bal\left\{
\BAL
L u_0&=\xr\tau\quad&&\text{in}\;\; \xO\\
v&=0\quad\quad &&\text{in}\;\; \BBR^N\setminus\xO,
\EAL\right.
\eal
satisfying $u_0\leq CW_{s,p}^{2\diam(\xO)}[\xr\tau].$ We may construct a nondecreasing sequence $\{u_n\}_{n\geq0},$ such that $u_n$ is a very weak solution to problem
\bal\left\{
\BAL
L u_n&=u_{n-1}^\xk+\xr\tau\quad&&\text{in}\;\; \xO\\
v&=0\quad\quad &&\text{in}\;\; \BBR^N\setminus\xO
\EAL\right.
\eal
and satisfies $0\leq u_n\leq CW_{s,p}^{2\diam(\xO)}[\xm_{n-1}]$ for any $n\in\BBN,$ where ${\dd}\xm_{n-1}=u_{n-1}^\xk {\dd} x+\xr{\dd}\tau.$
 In addition, by \eqref{16} and Proposition \ref{nwolffest} there holds

\ba\label{18}
C^{-1} W_{s,p}^{\frac{d(x)}{8}}[\xm_{n-1}](x)\leq u_n(x)\leq w(x)\quad \;\text{for a.e.}\;\;x\in\xO,
\ea
where the positive constant $C^{-1}$ depends only on $N,p,s,q.$ Finally, $u_n$ satisfies \eqref{estn1}-\eqref{est2} with ${\dd}\xm=w^\xk {\dd} x+\xr \dd\tau.$

Proceeding as in the proof of Proposition \eqref{sola}, we may show that there exists a subsequence, still denoted by $\{u_n\},$ such that $u_n\to u$  a.e. in $\BBR^N$ and $u$ is a very weak solution of problem \eqref{powersource}.  By \eqref{18} and Fatou's Lemma, we obtain estimate \eqref{estsourcewolff}. The proof is complete.
\end{proof}

\begin{proof}[\textbf{Proof of Theorem \ref{supercrpower1}}]
We will first prove that $(i)$ implies $(ii)$ by using some ideas from \cite{PVsource}. Without loss of generality we assume that $\xr=1.$ Extend $\xm$ to whole $\BBR^N$ by setting $\xm(\BBR^N\setminus \xO)=0.$

Let $0\leq g\in L^\frac{\xk}{p-1}(\BBR^N;\xm).$ We set
\bal
M_\xm g(x):=\sup_{r>0,\;\xm(B(x,r))\neq0}\xm(B(x,r))^{-1}\int_{B(x,r)}g(y){\dd} \xm.
\eal

It is well known that there exists a positive constant $c_1$ depending only on $N,p,\xk$ such that

\ba
\int_{\BBR^N} (M_\xm g(x))^\frac{\xk}{p-1}{\dd} \xm\leq c_1 \int_{\BBR^N} |g(x)|^\frac{\xk}{p-1}{\dd} \xm
\ea
(see, e.g., \cite{F}).
Also,

\ba\label{74}\BAL
\int_\xO&\left(W_{s,p}^{\frac{d(x)}{8}}[g\xm](x)\right)^\xk{\dd} x\leq \int_\xO\left(W_{s,p}^{\frac{d(x)}{8}}[\xm](x)\right)^\xk(M_\xm g(x))^\frac{\xk}{p-1}{\dd} x\\
&\leq C^\xk \int_{\xO}u^\xk(x) (M_\xm g(x))^\frac{\xk}{p-1}{\dd} x\leq C^\xk\int_\xO  (M_\xm g(x))^\frac{\xk}{p-1}{\dd} \xm\leq c_2\int_{\BBR^N} |g(x)|^\frac{\xk}{p-1}{\dd} \xm
\EAL
\ea

Let $K=\supp\tau.$ By the assumption, we have that $r_0:=\dist(K,\partial\xO)>0.$ Set $g=\1_K \tilde g,$ for any nonnegative $\tilde g\in L^\frac{\xk}{p-1}(\BBR^N;\xm_{\lfloor K}).$ We first note that $B_{\frac{r_0}{8}}(x)\cap K=\emptyset$ if $x\in\xO$ with $d(x)<\frac{r_0}{8}$ or  if $x\in\BBR^N\setminus \xO,$ which implies
\bal
W_{s,p}^{\frac{d(x)}{24}}[\tilde g\xm_{\lfloor K}](x)=0
\eal
if $x\in\xO$ with $d(x)<\frac{r_0}{24}$ or  if $x\in\BBR^N\setminus \xO.$ Therefore, by the above equality and \eqref{74}, we have
\bal
\int_{\BBR^N}\left(W_{s,p}^{\frac{r_0}{24}}[\tilde g\xm_{\lfloor K}](x)\right)^\xk {\dd} x\leq \int_\xO\left(W_{s,p}^{\frac{d(x)}{8}}[g\xm](x)\right)^\xk{\dd} x\leq c_2\int_{\BBR^N} |\tilde g(x)|^\frac{\xk}{p-1}{\dd} \xm_{\lfloor K}.
\eal

Also, by \cite[Theorem 2.3]{BNV} (see also \cite[Corollary 3.6.3]{Adbook}), we have

\ba\label{59}
\int_{\BBR^N}\left(W_{s,p}^{\frac{r_0}{24}}[\tilde g\xm_{\lfloor K}](x)\right)^\xk{\dd} x\approx\int_{\BBR^N}(\BBG_{sp}[\tilde g\xm_{\lfloor K}])^\frac{\xk}{p-1}{\dd} x,
\ea
where the implicit constant depends only on $s,p,N,\xk$ and $r_0.$

Hence, combining the last two displays, we may show that there exists a positive constant $c_3=c_3(N,p,s,\xk,r_0)$ such that

\ba\label{20}
\int_{\BBR^N}(\BBG_{sp}[\tilde g\xm_{\lfloor K}])^\frac{\xk}{p-1}{\dd} x \leq c_3\int_{\BBR^N} |\tilde g(x)|^\frac{\xk}{p-1}{\dd} \xm_{\lfloor K}.
\ea
Let $f\in L^{\frac{\xk}{\xk-p+1}}(\BBR^N).$ Then, for any $\tilde g\in L^\frac{\xk}{p-1}(\BBR^N;\xm_{\lfloor K}),$ there holds

\bal
\left|\int_{\BBR^N} f(x) G_{sp}*(\tilde g\xm_{\lfloor K})(x) {\dd} x\right|&=\left|\int_{\BBR^N}\tilde g(y) G_{sp}*f(y){\dd}\xm_{\lfloor K}\right|\\
&\leq C_1\norm{f}_{L^{\frac{\xk}{\xk-p+1}}(\BBR^N)}\norm{\tilde g}_{L^\frac{\xk}{p-1}(\BBR^N;\xm_{\lfloor K})}.
\eal
The last inequality implies,
\bal
 \int_{\BBR^N}|G_{sp}*f(x)|^{\frac{\xk}{\xk-p+1}}{\dd} \xm_{\lfloor K}\leq c_4\int_{\BBR^N}|f|^{\frac{\xk}{\xk-p+1}}{\dd} x\quad\forall f\in L^{\frac{\xk}{\xk-p+1}}(\BBR^N).
 \eal
By \cite[Theorem 7.2.1]{Adbook}, the above inequality is equivalent to
\ba\label{19}
\xm_{\lfloor K}(F)\leq c_5 \mathrm{Cap}_{{sp,\frac{\xk}{\xk-p+1}}}(F),
\ea
for any compact $F\subset\BBR^N.$ \eqref{caocon2} follows by the above inequality and the fact that $\tau\leq \xm_{\lfloor K}.$

Next, we prove that (ii) implies (iii). We note that proceeding as above, in the opposite direction, we may prove that \eqref{caocon2} implies

\bal
\int_{\BBR^N}(\BBG_{sp}[ \tilde g\tau])^\frac{\xk}{p-1}{\dd} x \leq c_3\int_{\BBR^N} |\tilde g(x)|^\frac{\xk}{p-1}{\dd} \tau,\quad\forall g\in L^\frac{\xk}{p-1}(\BBR^N;\tau).
\eal

By \eqref{59} and taking $\tilde g=\1_B,$ we can easily show that there exists a positive constant $C$ depending only on $N,s,p,\xO$ such that

\bal
\int_{\BBR^N} (W_{s,p}^{2\diam(\xO)}[\tau_{\lfloor B}])^\xk {\dd} x \leq C\tau(B).
\eal

We will show that (iii) implies (iv). Let $R=2\diam(\xO)$ and $C_3$ be the constant in \eqref{57}. In the spirit of the proof of \cite[Theorem 2.10]{PV}, we need to prove that there exists a positive constant $c_0=c_0(N,p,\xk,s,C_3, R,\tau(\xO))>0$ such that

\ba\label{anistau}
\tau(B_t(x))\leq c_0t^{\frac{\xk(N-sp)-N(p-1)}{\xk-p+1}}
\ea
for any $t\leq R$ and $\forall x\in \xO.$

Concerning the proof of the above inequality, we first note that for any $y\in B_t(x)$ and $t\leq \frac{R}{4},$ there holds

\bal
W_{s,p}^{R}[\tau_{\lfloor B_t(x)}](y)&=\int_0^{R}\left(\frac{\tau(B_r(y)\cap B_t(x))}{r^{N-sp}}\right)^{\frac{1}{p-1}}\frac{{\dd} r}{r}\geq \int_{2t}^{4t}\left(\frac{\tau(B_r(y)\cap B_t(x))}{r^{N-sp}}\right)^{\frac{1}{p-1}}\frac{{\dd} r}{r}\\
&\geq C(N,p,s)\left(\frac{\tau(B_t(x))}{t^{N-sp}}\right)^{\frac{1}{p-1}}.
\eal
By the above inequality, we deduce

\ba\label{58}\BAL
t^NC^\xk(N,p,s)\left(\frac{\tau(B_t(x))}{t^{N-sp}}\right)^{\frac{\xk}{p-1}}&\leq \int_{B_t(x)}(W_{s,p}^{R}[\tau_{\lfloor B_t(x)}](y))^\xk {\dd}  y\\
&\leq C_3\tau(B_t(x)),\quad\forall t\in(0\frac{R}{4}],
\EAL
\ea
where in the last inequality we used \eqref{57}. This implies \eqref{anistau}.

For any $x\in\xO$ and $t<R,$ we set

\bal
\xn_t(x):=\int_{B_t(x)}\bigg(\int_0^t\left(\frac{\tau(B_r(y))}{r^{N-sp}}\right)^{\frac{1}{p-1}}\frac{{\dd} r}{r}\bigg)^\xk{\dd} y
\eal
and
\bal
\xm_t(x):=\int_{B_t(x)}\bigg(\int_t^{R}\left(\frac{\tau(B_r(y))}{r^{N-sp}}\right)^{\frac{1}{p-1}}\frac{{\dd} r}{r}\bigg)^\xk{\dd} y.
\eal
Then we can easily prove that
\ba\label{69}
W_{s,p}^{R}[(W_{s,p}^{R}[\tau])^\xk]\leq C(q,p)\bigg(\int_0^{R}\left(\frac{\xn_t(x)}{t^{N-sp}}\right)^\frac{1}{p-1}\frac{{\dd} t}{t} +\int_0^{R}\left(\frac{\xm_t(x)}{t^{N-sp}}\right)^\frac{1}{p-1}\frac{{\dd} t}{t}\bigg).
\ea

Now, we treat the first term on the right hand in $\eqref{69}.$ By \eqref{57}, we have
\ba\label{70}\BAL
\xn_t(x)&=\int_{B_t(x)}\bigg(\int_0^t\left(\frac{\tau(B_r(y)\cap B_{2t}(x))}{r^{N-sp}}\right)^{\frac{1}{p-1}}\frac{{\dd} r}{r}\bigg)^\xk{\dd} y\leq C\tau(B_{2t}(x)),
\EAL
\ea
which implies

\ba\label{71}
\int_0^{R}\left(\frac{\xn_t(x)}{t^{N-sp}}\right)^\frac{1}{p-1}\frac{{\dd} t}{t}\leq CW_{s,p}^{2R}[\tau](x).
\ea

Next, we treat the second term on the right hand in $\eqref{69}.$ First we note that
\bal
\xm_t(x)&\leq
\int_{B_t(x)}\bigg(\int_t^{R}\left(\frac{\tau(B_{2r}(x))}{r^{N-sp}}\right)^{\frac{1}{p-1}}\frac{{\dd} r}{r}\bigg)^\xk{\dd} y\\
&\leq C(N) t^N\bigg(\int_t^{2R}\left(\frac{\tau(B_{r}(x))}{r^{N-sp}}\right)^{\frac{1}{p-1}}\frac{{\dd} r}{r}\bigg)^\xk\\
&=:C(N)t^N\xm_{1,t}^\xk(x),
\eal
which implies

\bal
&\int_0^{R}\left(\frac{\xm_t(x)}{t^{N-sp}}\right)^\frac{1}{p-1}\frac{{\dd} t}{t}\leq C(N,p) \int_0^{R}\xm_{1,t}^\frac{\xk}{p-1}(x)t^{\frac{sp}{p-1}-1}{\dd} t\\
&=C(N,p,s,q)\bigg(\xm_{1,R}^{\frac{\xk}{p-1}}(x)R^{\frac{sp}{p-1}}+
\int_0^{R}\left(\xm_{1,t}(x)\right)^{\frac{\xk}{p-1}-1}t^{\frac{sp}{p-1}}
\left(\frac{\tau(B_t(x))}{t^{N-sp}}\right)^{\frac{1}{p-1}}\frac{{\dd} t}{t}\bigg),
\eal
where we have used integration by parts in the last equality. By \eqref{anistau}, we have
\bal
\xm_{1,t}(x)=\int_t^{2R}\left(\frac{\tau(B_{r}(x))}{r^{N-sp}}\right)^{\frac{1}{p-1}}\frac{{\dd} r}{r}\leq C t^{-\frac{sp}{\xk-p+1}}.
\eal
Combining the last two displays, we obtain

\ba\label{72}
\int_0^{R}\left(\frac{\xm_t(x)}{t^{N-sp}}\right)^\frac{1}{p-1}\frac{{\dd} t}{t}\leq C(N,p,s,\xk,R)\bigg(\tau(B_{2R}(x))^{\frac{\xk}{(p-1)^2}}+W_{s,p}^{R}[\tau](x)\bigg).
\ea

The desired result follows by \eqref{69}, \eqref{71}, \eqref{72} and the fact that 
\bal
\tau(B_{2R}(x))^{\frac{\xk}{(p-1)^2}}&\leq \tau(\xO)^{\frac{\xk-p+1}{(p-1)^2}}\tau(B_{\frac{R}{2}}(x))^{\frac{1}{p-1}}\leq C(R,N,p,\tau,s,\xk)W_{s,p}^{R}[\tau](x)\quad\forall x\in\xO.
\eal

\end{proof}


\begin{thebibliography}{99}

\bibitem{Adbook}D.R. Adams, L.I. Hedberg, \emph{Function Spaces and Potential Theory,} in: Grundlehren Math. Wiss., vol. 314, Springer, 1996.

\bibitem{abde} B. Abdellaoui, A. Attar, R. Bentifour, \emph{On the fractional p-laplacian equations with weights and
general datum}, Adv. Nonlinear Anal., 8 (2019), 144–174. {\url{https://doi.org/10.1515/anona-2016-0072}}

\bibitem{BP} P. Baras, M. Pierre. Singularit\'es \'eliminables pour des \'equations semi-lin\'eaires.
Ann. Inst. Fourier 34, 117-135 (1984).

\bibitem{B} M.-F. Bidaut-V\'eron, \emph{Removable Singularities and Existence for a Quasilinear Equation with Absorption or Source Term and Measure Data}, Adv. Nonlinear Stud. 3 (2003) 25–63.

\bibitem{BNV} M.-F. Bidaut-V\'eron, Q.-H. Nguyen and L. V\'eron,\emph{Quasilinear Lane–Emden equations with absorption
and measure data}, J. Math. Pures Appl. 102 (2014) 315–337

\bibitem{BG} L. Boccardo - T. Gallou\"{e}t, \emph{Nonlinear elliptic and parabolic equations involving
measure data,} J. Funct. Anal., 87 (1989), pp. 149-169.


\bibitem{CFV} H. Chen, P. Felmer and L. V\'{e}ron, \emph{Elliptic equations involving general subcritical source nonlinearity and measures,} arXiv:1409.3067[math.AP]

\bibitem{CQ} H. Chen, A. Quaas, \emph{Classification of isolated singularities of nonnegative solutions to fractional semilinear elliptic
equations and the existence results,} J. Lond. Math. Soc. 97 (2) (2018) 196–221.


\bibitem{CV1} H. Chen, L. V\'{e}ron, \emph{Semilinear fractional elliptic equations involving measures,} J. Diferential Equations 257 (2014), no. 5,
1457–1486.


\bibitem{MMOP} G. Dal Maso, F.Murat, L.Orsina and A.Prignet, Renormalized solutions of elliptic equations with general
measure data, Ann. Sc. Norm. Super. Pisa, Cl. Sci. 28(1999) 741-808.

\bibitem{diCKP_local} A. Di Castro, T. Kuusi and G. Palatucci, \emph{Local behavior of fractional p-minimizers,} Ann. Inst. H. Poincar\'e Anal. Non
Lin\'eaire 33 (2016) 1279-1299.

\bibitem{diCKP_local2} A. Di Castro, T. Kuusi and G. Palatucci,\emph{ Nonlocal Harnack inequalities,} J. Funct. Anal. 267 (6) (2014)
1807–1836.



\bibitem{hitch} E. Di Nezza, G. Palatucci, E. Valdinoci, \emph{Hitchhiker’s guide to the fractional Sobolev spaces.} Bull. Sci.
Math. 136, 521–573 (2012)

\bibitem{F} R. Fefferman, \emph{Strong differentiation with respect to measures,} Amer. J. Math. 103 (1981) 33–40.

\bibitem{egw} K. T. Gkikas, \emph{Quasilinear elliptic equations involving measure valued absorption terms and measure data,} preprint.

\bibitem{GkiNg_source} K. Gkikas and P.-T. Nguyen, \emph{Semilinear elliptic Schr\"odinger equations involving singular potentials and source terms}, preprint.

\bibitem{HKM} J. Heinonen, T. Kilpel\"{a}inen, and O. Martio,\emph{ Nonlinear potential theory of de-
generate elliptic equations.} Oxford Mathematical Monographs. The Clarendon Press
Oxford University Press, New York, 1993. , Oxford Science Publications.

\bibitem{IMS} A. Iannizzotto, S. Mosconi, and M. Squassina, \emph{Global H\"older regularity for the fractional p-Laplacian,}
Rev. Mat. Iberoam., 32(4):1353–1392, 2016.

\bibitem{KM1} T. Kilpel\"{a}inen, J. Mal\'{y}, \emph{The Wiener test and potential estimates for quasilinear elliptic equations,} Acta Math. 172
(1994) 137-161.


\bibitem{KM2} T. Kilpel\"{a}inen, J. Mal\'{y}, \emph{J. Maly, Degenerate elliptic equations with measure data and nonlinear
potentials.} Ann. Scuola Norm. Sup. Pisa Cl. Sci. (4), 19(4):591–613, 1992.

\bibitem{KLL} M. Kim, K.-A. Lee, and S.-C. Lee, \emph{The Wiener criterion for nonlocal
Dirichlet problems,} Commun. Math. Phys. 400, 1961–2003 (2023)

\bibitem{KKL}J. Korvenp\"a\"a, T. Kuusi, E. Lindgren. \emph{Equivalence of solutions to fractional p-Laplace type equations,} J. Math.
Pures Appl. (9) 132, 1–26, (2019).

\bibitem{KKP} J. Korvenp\"a\"a, T. Kuusi, and G. Palatucci, \emph{Fractional superharmonic functions and the Perron method
for nonlinear integro-differential equations.} Math. Ann., 369(3-4):1443–1489, 2017.


\bibitem{KKP2} J. Korvenp\"a\"a, T. Kuusi, and G. Palatucci, \emph{The obstacle problem for nonlinear integro-differential operators,} Calc. Var.
Partial Differential Equations 55 (3) (2016) 63, 29 pp.

\bibitem{KMS2} T. Kuusi, G. Mingione and Y. Sire, \emph{Regularity issues involving the fractional $p$-Laplacian.} Recent developments in nonlocal theory, 303-334, De Gruyter, Berlin, 2018.

\bibitem{KMS} T. Kuusi, G. Mingione, Y. Sire, \emph{Nonlocal equations with measure data,} Comm. Math. Phys. 337 (2015) 1317–1368.

\bibitem{LL} E. Lindgren, P. Lindqvist, \emph{Perron’s method and Wiener’s theorem for a nonlocal equation,} Potential Anal. 46 (2017)
705–737.

\bibitem{LL2} E. Lindgren, P. Lindqvist, \emph{Fractional eigenvalues,} Calc. Var. Partial Differential Equations 49(1-2),
795–826 (2014)

\bibitem{P} G. Palatucci, \emph{The dirichlet problem for the p-fractional laplace equation,} Nonlinear Anal. 177 (part b) (2018) 699–732.

\bibitem{PVsource} N. C. Phuc, I. E. Verbitsky, \emph{Singular quasilinear and Hessian equations and inequalities,} J. Funct. Anal. 256 (2009)
1875-1906.

\bibitem{PV} N. C. Phuc, I. E. Verbitsky, \emph{Quasilinear and Hessian equations of Lane–Emden type,} Ann. Math. 168(2008), 859–914.

\bibitem{Vbook} L. V\'{e}ron, \emph{Local and global aspects of quasilinear degenerate elliptic equations. Quasilinear
elliptic singular problems.} World Scientific Publishing Co. Pte. Ltd., Hackensack, NJ (2017).
xv+ pp. 1-457.
\end{thebibliography}
\end{document}